\documentclass{svproc}
\usepackage{amsmath, verbatim, amsfonts, amssymb, color}
\usepackage[mathscr]{euscript}
\usepackage{dsfont}

\usepackage{bm}
\usepackage{enumitem}
\usepackage{graphicx}
\definecolor{darkgreen}{rgb}{0,0.75,0}
\definecolor{darkred}{rgb}{0.75,0,0}
\definecolor{darkmagenta}{rgb}{0.5,0,0.5}
\usepackage{hyperref}
\hypersetup{colorlinks=true,citecolor=darkgreen,linkcolor=darkred,urlcolor=darkmagenta}

\spnewtheorem{thm}{Theorem}[section]{\bfseries}{\itshape}
\spnewtheorem{conj}[thm]{Conjecture}{\bfseries}{\itshape}
\spnewtheorem{cor}[thm]{Corollary}{\bfseries}{\itshape}
\spnewtheorem{lem}[thm]{Lemma}{\bfseries}{\itshape}
\spnewtheorem{prop}[thm]{Proposition}{\bfseries}{\itshape}
\spnewtheorem{dfn}[thm]{Definition}{\bfseries}{\rmfamily}
\spnewtheorem{ass}[thm]{Assumption}{\bfseries}{\rmfamily}
\spnewtheorem{exmp}[thm]{Example}{\bfseries}{\rmfamily}
\spnewtheorem{prb}[thm]{Problem}{\bfseries}{\rmfamily}
\spnewtheorem{rmk}[thm]{Remark}{\bfseries}{\rmfamily}
\spnewtheorem*{ntn}{Notation}{\bfseries}{\rmfamily}

\numberwithin{equation}{section}
\renewcommand{\figurename}{Figure\hspace*{1pt}}

\makeatletter
\def\@makefnmark{\hbox{(\@textsuperscript{\normalfont\@thefnmark})}}
\makeatother

\makeatletter
\@namedef{subjclassname@2020}{%
  \textup{2020} Mathematics Subject Classification}
\makeatother

\newcommand{\mr}[1]{\texttt{\href{http://www.ams.org/mathscinet-getitem?mr=#1}{MR#1}}}
\newcommand{\arxiv}[1]{\texttt{\href{http://arxiv.org/abs/#1}{arXiv:#1}}}
\newcommand{\refmet}{d}
\newcommand{\refmeas}{m}
\newcommand{\one}{\mathds{1}} 

\newcommand{\diam}{\mathop{\operatorname{diam}}}
\newcommand{\id}{\mathop{\operatorname{id}}}
\newcommand{\unif}{\mathop{\operatorname{unif}}}
\newcommand{\supp}{\mathop{\operatorname{supp}}}

\renewcommand{\ackname}{Acknowledgements.}%
\providecommand{\ack}[1]{\par\addvspace\baselineskip
\noindent\ackname\enspace\ignorespaces#1}%
\def\subjclassname{\textup{2020} \textit{Mathematics Subject Classification:}}%
\providecommand{\subjclass}[1]{\par\addvspace\baselineskip
\noindent\subjclassname\enspace\ignorespaces#1}%

\begin{document}
\mainmatter
\title{\bfseries On singularity of energy measures for symmetric diffusions with full off-diagonal heat kernel estimates II: Some borderline examples}
\titlerunning{On singularity of energy measures for diffusions II: Borderline examples}

\author{\begin{picture}(0,0)\put(-139,96){Version of November 4, 2021}\end{picture}Naotaka Kajino}
\authorrunning{N. Kajino}
\tocauthor{David Berger, Franziska K\"uhn and Ren\'e L. Schilling}
\institute{Research Institute for Mathematical Sciences, Kyoto University\\
\email{nkajino@kurims.kyoto-u.ac.jp}}

\maketitle

\begin{abstract}
We present a concrete family of fractals, which we call the
\emph{(two-dimensional) thin scale irregular Sierpi\'{n}ski gaskets} and each of
which is equipped with a canonical strongly local regular symmetric Dirichlet form.
We prove that any fractal $K$ in this family satisfies the full off-diagonal heat
kernel estimates with some space-time scale function $\Psi_{K}$ and the singularity
of the associated energy measures with respect to the canonical volume measure
(uniform distribution) on $K$, and also that the decay rate of $r^{-2}\Psi_{K}(r)$
to $0$ as $r\downarrow 0$ can be made arbitrarily slow by suitable choices of $K$.
These results together support the energy measure singularity dichotomy conjecture
[\emph{Ann.\ Probab.}\ \textbf{48} (2020), no.~6, 2920--2951, Conjecture 2.15]
stating that, if the full off-diagonal heat kernel estimates with space-time scale
function $\Psi$ satisfying $\lim_{r\downarrow 0}r^{-2}\Psi(r)=0$ hold for a strongly
local regular symmetric Dirichlet space with complete metric, then the associated energy
measures are singular with respect to the reference measure of the Dirichlet space.
\ack{The author would like to thank Martin T.\ Barlow for his valuable suggestion
in \cite{Bar19} of the family of fractals studied in this paper as possible examples
to examine the validity of the energy measure singularity dichotomy conjecture
(Conjecture \ref{conj:HK-singular} below).
The author was supported in part by JSPS KAKENHI Grant Number JP18H01123.}
\end{abstract}

\section{Introduction}\label{sec:intro}

This paper is a follow-up of the author's recent joint work \cite{KM} with Mathav Murugan
on singularity of energy measures associated with a strongly local regular symmetric
Dirichlet space $(K,\refmet,\refmeas,\mathcal{E},\mathcal{F})$ satisfying full
off-diagonal heat kernel estimates. The \emph{$\mathcal{E}$-energy measure}
$\mu_{\langle u\rangle}$ of $u\in\mathcal{F}$ is a Borel measure on $K$ which plays,
in the theory of regular symmetric Dirichlet forms as presented in \cite{FOT,CF}, the
same roles as the classical energy integral measure $|\nabla u|^{2}\,dx$ on $\mathbb{R}^{N}$.
It is defined for $u\in\mathcal{F}\cap L^{\infty}(K,\refmeas)$ as the unique
Borel measure on $K$ such that
\begin{equation}\label{eq:EnergyMeas-intro}
\int_{K}f\,d\mu_{\langle u\rangle}=\mathcal{E}(u,fu)-\frac{1}{2}\mathcal{E}(u^{2},f)
	\qquad\textrm{for any $f\in\mathcal{F}\cap\mathcal{C}_{\mathrm{c}}(K)$,}
\end{equation}
where $\mathcal{C}_{\mathrm{c}}(K)$ denotes the space of $\mathbb{R}$-valued
continuous functions on $K$ with compact supports, and then for $u\in \mathcal{F}$ by
$\mu_{\langle u\rangle}(A):=\lim_{n\to\infty}\mu_{\langle(-n)\vee(u\wedge n)\rangle}(A)$
for each Borel subset $A$ of $K$; see
\cite[Theorem 1.4.2-(ii),(iii), (3.2.13), (3.2.14) and (3.2.15)]{FOT}
for the details of this definition.

The main results of \cite{KM} concern the singularity and the absolute continuity of
the $\mathcal{E}$-energy measures $\mu_{\langle u\rangle}$ with respect to the reference
measure $\refmeas$. While $\mu_{\langle u\rangle}$ is easily identified as
$\langle\nabla u,\nabla u\rangle_{x}\,d\refmeas(x)$ if
$\mathcal{E}=\int_{K}\langle\nabla\cdot,\nabla\cdot\rangle_{x}\,d\refmeas(x)$
for some linear differential operator $\nabla$ satisfying the Leibniz rule and some measurable field
$\langle\cdot,\cdot\rangle_{x}$ of non-negative definite symmetric bilinear forms,
there is no simple expression of $\mu_{\langle u\rangle}$ and the nature of
$\mu_{\langle u\rangle}$ is a deep mystery when $K$ is a fractal. The question of
whether $\mu_{\langle u\rangle}$ is singular with respect to $\refmeas$ is probably
the most fundamental one toward better understanding of $\mu_{\langle u\rangle}$ in
such cases, had been answered affirmatively for essentially all known examples of
self-similar Dirichlet forms on self-similar fractals in \cite{Kus89,Kus93,BST,Hin05,HN},
but had been studied only under the assumption of exact self-similarity until \cite{KM}.
As the main results of \cite{KM}, it has been now proved that the $\mathcal{E}$-energy
measures $\mu_{\langle u\rangle}$ are singular or absolutely continuous with respect to
$\refmeas$ according to whether the behavior of the associated heat kernel $p_{t}(x,y)$
in infinitesimal scale is \emph{``sufficiently sub-Gaussian''} or \emph{``Gaussian''},
as stated in the following theorem.
Recall that a family $\{p_{t}\}_{t\in(0,\infty)}$ of $[-\infty,\infty]$-valued
Borel measurable functions on $K\times K$ is called a \emph{heat kernel} of
$(K,\refmet,\refmeas,\mathcal{E},\mathcal{F})$ if and only if
the symmetric Markovian semigroup $\{T_{t}\}_{t\in(0,\infty)}$ on $L^{2}(K,\refmeas)$
associated with $(\mathcal{E},\mathcal{F})$ is expressed as
$T_{t}u=\int_{K}p_{t}(\cdot,y)u(y)\,d\refmeas(y)$ $\refmeas$-a.e.\ for any $t\in(0,\infty)$
and any $u\in L^{2}(K,\refmeas)$. We set $\diam(K,\refmet):=\sup_{x,y\in K}\refmet(x,y)$
and $B_{\refmet}(x,r):=\{y\in K\mid \refmet(x,y)<r\}$ for $(x,r)\in K\times(0,\infty)$.
\begin{thm}[A simplification of {\cite[Theorem 2.13]{KM}}]\label{thm:HK-singular}
Let $(K,\refmet,\refmeas,\mathcal{E},\mathcal{F})$ be a \emph{metric measure Dirichlet space},
i.e., a strongly local regular symmetric Dirichlet space with $(K,\refmet)$ complete and
$K$ containing at least two elements, so that $\diam(K,\refmet)\in(0,\infty]$.
Let $\Psi\colon[0,\infty)\to[0,\infty)$ be a homeomorphism satisfying
\begin{equation}\label{eq:Psi-ass}
c_{\Psi}^{-1}\biggl(\frac{R}{r}\biggr)^{\beta_{0}}\leq\frac{\Psi(R)}{\Psi(r)}
	\leq c_{\Psi}\biggl(\frac{R}{r}\biggr)^{\beta_{1}}
	\quad\textrm{for any $r,R\in(0,\infty)$ with $r\leq R$}
\end{equation}
for some $c_{\Psi},\beta_{0},\beta_{1}\in[1,\infty)$ with $1<\beta_{0}\leq\beta_{1}$, and define
$\Phi_{\Psi}\colon[0,\infty)\to[0,\infty)$ by $\Phi_{\Psi}(s):=\sup_{r\in(0,\infty)}(s/r-1/\Psi(r))$.
Suppose further that $(K,\refmet,\refmeas,\mathcal{E},\mathcal{F})$ satisfies the
\emph{full off-diagonal heat kernel estimates} \hypertarget{fHKE}{$\textup{fHKE}(\Psi)$},
i.e., that there exist a heat kernel $\{p_{t}\}_{t\in(0,\infty)}$ of
$(K,\refmet,\refmeas,\mathcal{E},\mathcal{F})$ and $c_{1},c_{2},c_{3},c_{4}\in(0,\infty)$ such that
\begin{equation}\tag*{$\textup{fHKE}(\Psi)$}\label{eq:fHKE}
\frac{c_{1}\exp\bigl(-c_{2}t\Phi_{\Psi}(\refmet(x,y)/t)\bigr)}{m\bigl(B_{\refmet}(x,\Psi^{-1}(t))\bigr)}\leq p_{t}(x,y)
	\leq\frac{c_{3}\exp\bigl(-c_{4}t\Phi_{\Psi}(\refmet(x,y)/t)\bigr)}{m\bigl(B_{\refmet}(x,\Psi^{-1}(t))\bigr)}
\end{equation}
for $\refmeas$-a.e.\ $x,y\in K$ for each $t\in(0,\infty)$. Then the following hold:
\begin{enumerate}[label=\textup{(\arabic*)},align=left,leftmargin=*]
\item\label{it:case-sing}\textup{(\ref{eq:fHKE} with ``sufficiently sub-Gaussian'' $\Psi$ implies singularity)} If
	\begin{equation}\label{eq:case-sing}
	\liminf_{\lambda\to\infty}\liminf_{r\downarrow 0}\frac{\lambda^{2}\Psi(r/\lambda)}{\Psi(r)}=0,
	\end{equation}
	then $\mu_{\langle u\rangle}$ is singular with respect to $\refmeas$ for any $u\in\mathcal{F}$.
\item\label{it:case-ac}\textup{(\ref{eq:fHKE} with ``Gaussian'' $\Psi$ implies absolute continuity)} If
	\begin{equation}\label{eq:case-ac}
	\limsup_{r\downarrow 0}\frac{\Psi (r)}{r^{2}}>0,
	\end{equation}
	then $\refmeas(A)=0$ if and only if $\sup_{u\in\mathcal{F}}\mu_{\langle u\rangle}(A)=0$
	for each Borel subset $A$ of $K$, thus in particular $\mu_{\langle u\rangle}$ is
	absolutely continuous with respect to $\refmeas$ for any $u\in\mathcal{F}$, and
	there exist $r_{1}\in(0,\diam(K,\refmet))$ and $c_{5}\in[1,\infty)$ such that
	\begin{equation}\label{eq:Psi-Gauss}
	c_{5}^{-1}r^{2}\leq\Psi(r)\leq c_{5}r^{2}\qquad\textrm{for any $r\in(0,r_{1})$.}
	\end{equation}
\end{enumerate}
\end{thm}
\begin{rmk}\label{rmk:HK-singular}
Let $\Psi\colon[0,\infty)\to[0,\infty)$ be a homeomorphism satisfying \eqref{eq:Psi-ass}
for some $c_{\Psi},\beta_{0},\beta_{1}\in[1,\infty)$ with $1<\beta_{0}\leq\beta_{1}$,
and let $(K,\refmet,\refmeas,\mathcal{E},\mathcal{F})$ be a metric measure Dirichlet
space satisfying \ref{eq:fHKE}.
\begin{enumerate}[label=\textup{(\arabic*)},align=left,leftmargin=*]
\item\label{it:HK-singular-assump}It is known that in this situation
	$(K,\refmet,\refmeas,\mathcal{E},\mathcal{F})$ satisfies the assumptions of
	\cite[Theorem 2.13]{KM}, namely $\textup{VD}$, $\textup{PI}(\Psi)$,
	$\textup{CS}(\Psi)$ and the chain condition for $(K,\refmet)$.
	Indeed, $\textup{VD}$ follows in the same way as \cite[Proof of Lemma 5.1-(i)]{BGK}
	by integrating the lower inequality in \ref{eq:fHKE} on $B_{\refmet}(x,2\Psi^{-1}(t))$
	with respect to $\refmeas$ and applying the upper bound on $\Phi_{\Psi}(R,t):=t\Phi_{\Psi}(R/t)$
	in \cite[(5.13)]{GK}, \eqref{eq:Psi-ass} and the inequality
	$\int_{B_{\refmet}(x,2\Psi^{-1}(t))}p_{t}(x,y)\,d\refmeas(y)\leq\int_{K}p_{t}(x,y)\,d\refmeas(y)\leq 1$
	for $\refmeas$-a.e.\ $x\in K$. Then $\textup{VD}$ and \ref{eq:fHKE} imply
	$\textup{PI}(\Psi)$ and $\textup{CS}(\Psi)$ by the results in \cite{BB04,BBK,AB,GHL}
	as summarized in \cite[Theorem 3.2]{Lie} and \cite[Theorem 2.8 and Remark 2.9]{KM},
	and \ref{eq:fHKE} also implies the chain condition for $(K,\refmet)$ by \cite[Theorem 2.11]{Mur}.
\item\label{it:fHKE-stable-Psi}\emph{If $\Psi_{0}\colon[0,\infty)\to[0,\infty)$ is a homeomorphism and
	$\Psi_{0}(r)/\Psi(r)\in[c_{0}^{-1},c_{0}]$ for any $r\in(0,\infty)$ for some $c_{0}\in[1,\infty)$, then
	$(K,\refmet,\refmeas,\mathcal{E},\mathcal{F})$ satisfies \hyperlink{fHKE}{$\textup{fHKE}(\Psi_{0})$}.}
	Indeed, this is immediate from \ref{eq:fHKE}, $\textup{VD}$, which is implied by \ref{eq:fHKE}
	as noted in \ref{it:HK-singular-assump} above, and the elementary observation based on
	\eqref{eq:Psi-ass} that
	$\Phi_{\Psi_{0}}(s)/\Phi_{\Psi}(s)\in\bigl[(c_{0}c_{\Psi})^{-\frac{1}{\beta_{0}-1}},(c_{0}c_{\Psi})^{\frac{1}{\beta_{0}-1}}\bigr]$
	for any $s\in(0,\infty)$.
\end{enumerate}
\end{rmk}

Note that, if $\Psi(r)=r^{\beta}$ for any $r\in[0,\infty)$ for some $\beta\in(1,\infty)$,
then $\Phi(s)=\beta^{-\frac{\beta}{\beta-1}}(\beta-1)s^{\frac{\beta}{\beta-1}}$
for any $s\in[0,\infty)$, so that \ref{eq:fHKE} with this $\Psi$ is the
typical form of heat kernel estimates known to hold widely; see, e.g.,
\cite{Stu95a,Stu96,SC,Gri} and references therein for the studies on the case of
$\beta=2$ and \cite{BP,Kum,FHK,BB92,BB99} for known results with $\beta>2$
for self-similar fractals. For this class of $\Psi$, the classification by
\eqref{eq:case-sing} and \eqref{eq:case-ac} gives a complete dichotomy between
$\beta>2$ and $\beta\leq 2$, with the latter identified further as $\beta=2$
by \eqref{eq:Psi-Gauss}. On the other hand, \eqref{eq:case-sing} and \eqref{eq:case-ac}
do not give a complete classification of general $\Psi$ since there are examples of
$\Psi$, like $\Psi(r)=r^{2}/\log(e-1+r^{-1})$, satisfying \eqref{eq:Psi-ass} but
neither \eqref{eq:case-sing} nor \eqref{eq:case-ac}, and it is not clear under \ref{eq:fHKE}
with such $\Psi$ whether the $\mathcal{E}$-energy measures $\mu_{\langle u\rangle}$
are singular or absolutely continuous with respect to the reference measure $\refmeas$.
In view of Theorem \ref{thm:HK-singular}, one might expect the following conjecture to hold.
\begin{conj}[Energy measure singularity dichotomy; a simplification of {\cite[Conjecture 2.15]{KM}}]\label{conj:HK-singular}
Theorem \textup{\ref{thm:HK-singular}}-\ref{it:case-sing} with \eqref{eq:case-sing} replaced by 
\begin{equation}\label{eq:case-nonGauss}
\lim_{r\downarrow 0}\frac{\Psi(r)}{r^{2}}=0
\end{equation}
\textup{(\ref{eq:fHKE} with ``however weakly sub-Gaussian'' $\Psi$ implies singularity)} holds.
\end{conj}

As announced already in \cite[Remark 2.14]{KM}, this paper is aimed at
giving a clear evidence that Conjecture \ref{conj:HK-singular} should be true,
by presenting concrete examples of metric measure Dirichlet spaces satisfying
both the singularity of the energy measures and \ref{eq:fHKE} for some $\Psi$,
\emph{whose decay rate at $0$ can be made arbitrarily close to $r^{2}$}.
Their state spaces are certain fractals, which we call the
\emph{(two-dimensional) thin scale irregular Sierpi\'{n}ski gaskets}
(see Figure \ref{fig:thinSISG} below), obtained by modifying the construction of the scale
irregular (or homogeneous random) Sierpi\'{n}ski gaskets studied in \cite{Ham92,BH,Ham00}
(see also \cite[Chapter 24]{Kig12}) so as to make them look very much like one-dimensional
frames in infinitesimal scale.
An arbitrarily slow decay rate of $\Psi(r)/r^{2}$ as $r\downarrow 0$ can be then realized
by choosing suitably the parameters defining the fractal to make its infinitesimal geometry
arbitrarily close to being one-dimensional, which is an idea suggested to the author by
Martin T.\ Barlow in \cite{Bar19}. An important point here is to allow \emph{infinitely}
many patterns of cell subdivisions to be present in the construction of the fractal,
in contrast to that of the usual scale irregular Sierpi\'{n}ski
gaskets considered in \cite{Ham92,BH,Ham00,Kig12}, each of which involves only
finitely many patterns of cell subdivisions and typically falls within the scope of
Theorem \ref{thm:HK-singular}-\ref{it:case-sing} as illustrated in \cite[Section 5]{KM}.
We remark that the singularity of the energy measures has been proved also in \cite{HY}
for a class of (two-dimensional) spatially inhomogeneous Sierpi\'{n}ski gaskets,
which typically do not satisfy the volume doubling property $\textup{VD}$
and are thereby beyond the scope of \cite[Theorem 2.13]{KM}.

The rest of this paper is organized as follows. In Section \ref{sec:tsisg}
we define the thin scale irregular Sierpi\'{n}ski gaskets and construct the canonical
Dirichlet forms (resistance forms) on them, and we verify in Section \ref{sec:tsisg-fHKE}
that they satisfy \ref{eq:fHKE} with $\Psi$ explicit in terms of their defining parameters
(Theorem \ref{thm:tsisg-fHKE}). In Section~\ref{sec:tsisg-sing} we prove the singularity
of the energy measures for the canonical Dirichlet form on \emph{any} thin scale irregular
Sierpi\'{n}ski gasket (Theorem \ref{thm:tsisg-sing}), and Section \ref{sec:realize-given-Psi}
is devoted to stating and proving our last main result that an arbitrarily slow decay rate
of $\Psi(r)/r^{2}$ can be realized by some thin scale irregular Sierpi\'{n}ski gasket
(Theorem \ref{thm:realize-given-Psi} and Proposition \ref{prop:tsisg-Psi-arbitrarily-slow}).
\begin{ntn}
In this paper, we adopt the following notation and conventions.
\begin{enumerate}[label=\textup{(\arabic*)},align=left,leftmargin=*]
\item The symbols $\subset$ and $\supset$ for set inclusion
	\emph{allow} the case of the equality.
\item $\mathbb{N}:=\{n\in\mathbb{Z}\mid n>0\}$, i.e., $0\not\in\mathbb{N}$.
\item The cardinality (the number of elements) of a set $A$ is denoted by $\#A$.
\item We set $a\vee b:=\max\{a,b\}$, $a\wedge b:=\min\{a,b\}$, $a^{+}:=a\vee 0$,
	$a^{-}:=-(a\wedge 0)$ and $\lfloor a\rfloor:=\max\{n\in\mathbb{Z}\mid n\leq a\}$
	for $a,b\in\mathbb{R}$, and we use the same notation also for $\mathbb{R}$-valued
	functions and equivalence classes of them. All numerical functions
	in this paper are assumed to be $[-\infty,\infty]$-valued.
\item The Euclidean inner product and norm on $\mathbb{R}^{2}$ are denoted by
	$\langle\cdot,\cdot\rangle$ and $|\cdot|$, respectively.
\item Let $K$ be a non-empty set. We define $\id_{K}\colon K\to K$ by $\id_{K}(x):=x$,
	$\one_{A}=\one_{A}^{K}\in\mathbb{R}^{K}$ for $A\subset K$ by
	$\one_{A}(x):=\one_{A}^{K}(x):=\bigl\{\begin{smallmatrix}1 & \textrm{if $x\in A$,}\\ 0 & \textrm{if $x\not\in A$,}\end{smallmatrix}$
	set $\one_{x}:=\one_{x}^{K}:=\one_{\{x\}}$ for $x\in K$
	and $\|u\|_{\sup}:=\|u\|_{\sup,K}:=\sup_{x\in K}|u(x)|$ for $u\colon K\to\mathbb{R}$.
\item Let $K$ be a topological space.
	We set $\mathcal{C}(K):=\{u\in\mathbb{R}^{K}\mid\textrm{$u$ is continuous}\}$, and
	the closure of $K\setminus u^{-1}(0)$ in $K$ is denoted by $\supp_{K}[u]$ for each
	$u\in\mathcal{C}(K)$. The Borel $\sigma$-algebra of $K$ is denoted by $\mathscr{B}(K)$.
\item Let $(K,\refmet)$ be a metric space. We set
	$B_{\refmet}(x,r):=\{y\in K\mid \refmet(x,y)<r\}$ for $(x,r)\in K\times(0,\infty)$.
\item Let $(K,\mathscr{B})$ be a measurable space and let $\mu,\nu$ be measures on
	$(K,\mathscr{B})$. We write $\nu \ll \mu$ and $\nu \perp \mu$ to mean that
	$\nu$ is absolutely continuous and singular, respectively, with respect to $\mu$.
\end{enumerate}
\end{ntn}
%
\section{The examples: Thin scale irregular Sierpi\'{n}ski gaskets}\label{sec:tsisg}
In this section, we introduce the (two-dimensional) thin scale irregular Sierpi\'{n}ski
gaskets, and construct the canonical Dirichlet forms (resistance forms) on them
by applying the standard method developed in \cite[Chapters 2 and 3]{Kig01}.
We closely follow \cite[Section 5]{KM} for the presentation of this section.
\begin{figure}[b]\centering
\includegraphics[height=69pt]{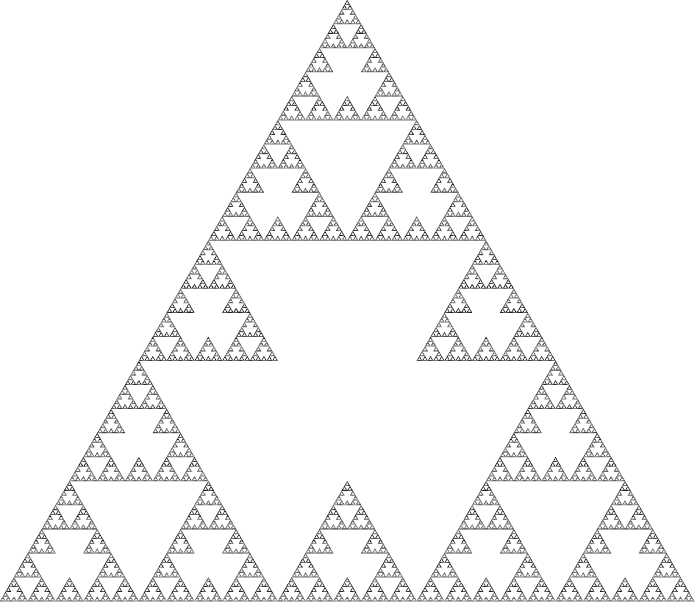}\quad
\includegraphics[height=69pt]{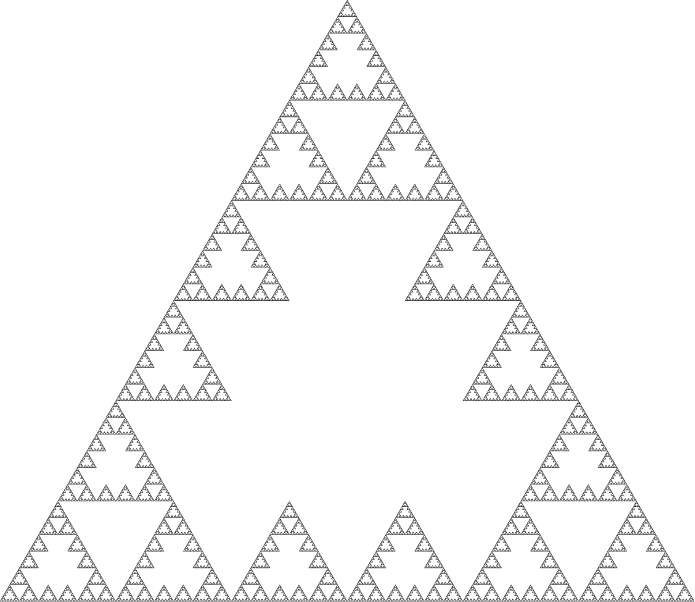}\quad
\includegraphics[height=69pt]{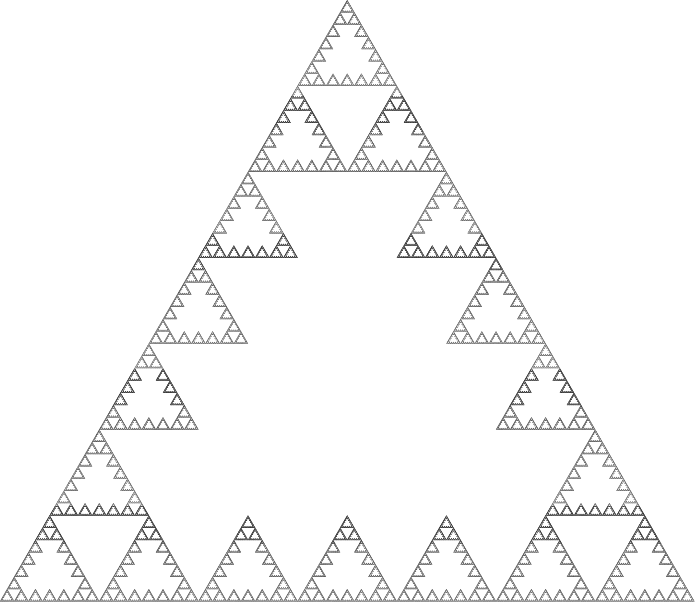}\quad
\includegraphics[height=69pt]{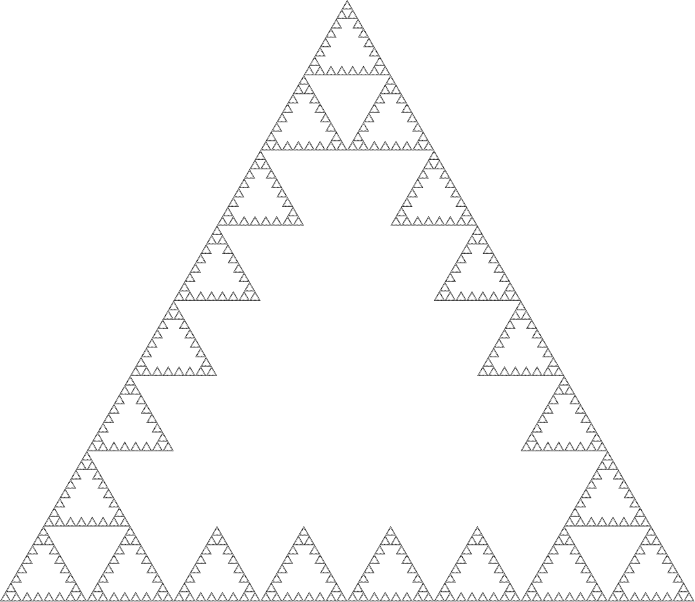}
\caption{The level-$l$ (self-similar) thin Sierpi\'{n}ski gaskets $K^{l}$ ($l=5,6,7,8$)}
\label{fig:thinSGs}
\end{figure}

To start with, the thin scale irregular Sierpi\'{n}ski gaskets are defined as follows.%
\begin{dfn}[Thin scale irregular Sierpi\'{n}ski gasket]\label{dfn:tsisg}
Let $q_{0},q_{1},q_{2}\in\mathbb{R}^{2}$ satisfy $|q_{j}-q_{k}|=1$ for any $j,k\in\{0,1,2\}$
with $j\not=k$, so that the convex hull $\triangle$ of $V_{0}:=\{q_{0},q_{1},q_{2}\}$
in $\mathbb{R}^{2}$ is a closed equilateral triangle with side length $1$.
For each $l\in\mathbb{N}\setminus\{1,2,3,4\}$, we set
\begin{equation}\label{eq:tlsg}
S_{l}:=\bigl\{(i_{1},i_{2})\in(\mathbb{N}\cup\{0\})^{2}
	\bigm|\textrm{$i_{1}+i_{2}\leq l-1$, $i_{1}i_{2}(l-1-i_{1}-i_{2})=0$}\bigr\},
\end{equation}
and for each $i=(i_{1},i_{2})\in S_{l}$ set
$q^{l}_{i}:=q_{0}+\sum_{k=1}^{2}(i_{k}/l)(q_{k}-q_{0})$ and define
$f^{l}_{i}\colon\mathbb{R}^{2}\to\mathbb{R}^{2}$ by $f^{l}_{i}(x):=q^{l}_{i}+l^{-1}(x-q_{0})$.
Let $\bm{l}=(l_{n})_{n=1}^{\infty}\in(\mathbb{N}\setminus\{1,2,3,4\})^{\mathbb{N}}$,
set $W^{\bm{l}}_{n}:=\prod_{k=1}^{n}S_{l_{k}}$ for $n\in\mathbb{N}\cup\{0\}$,
$W^{\bm{l}}_{*}:=\bigcup_{n=0}^{\infty}W^{\bm{l}}_{n}$, $|w|:=n$ and
$f^{\bm{l}}_{w}:=f^{l_{1}}_{w_{1}}\circ\cdots\circ f^{l_{n}}_{w_{n}}$
for $n\in\mathbb{N}\cup\{0\}$ and $w=w_{1}\ldots w_{n}\in W^{\bm{l}}_{n}$,
where $W^{\bm{l}}_{0}$ is defined as the singleton $\{\emptyset\}$ of the empty word
$\emptyset$ and $f^{\bm{l}}_{\emptyset}:=\id_{\mathbb{R}^{2}}$.
Noting that $\bigl\{\bigcup_{w\in W^{\bm{l}}_{n}}f^{\bm{l}}_{w}(\triangle)\bigr\}_{n=0}^{\infty}$
is a strictly decreasing sequence of non-empty compact subsets of $\triangle$, we define the
\emph{(two-dimensional) level-$\bm{l}$ thin scale irregular Sierpi\'{n}ski gasket}
$K^{\bm{l}}$ as the non-empty compact subset of $\triangle$ given by
\begin{equation}\label{eq:tsisg}
K^{\bm{l}}:=\bigcap_{n=0}^{\infty}\bigcup_{w\in W^{\bm{l}}_{n}}f^{\bm{l}}_{w}(\triangle)
\end{equation}
(see Figure \ref{fig:thinSISG}), and set $K^{\bm{l}}_{w}:=K^{\bm{l}}\cap f^{\bm{l}}_{w}(\triangle)$
and $F^{\bm{l}}_{w}:=f^{\bm{l}}_{w}|_{K^{\bm{l}^{|w|}}}$ for $w\in W^{\bm{l}}_{*}$,
where $\bm{l}^{k}:=(l_{n+k})_{n=1}^{\infty}$ for $k\in\mathbb{N}\cup\{0\}$.
We also set $V^{\bm{l}}_{n}:=\bigcup_{w\in W^{\bm{l}}_{n}}f^{\bm{l}}_{w}(V_{0})$ for
$n\in\mathbb{N}\cup\{0\}$ and $V^{\bm{l}}_{*}:=\bigcup_{n=0}^{\infty}V^{\bm{l}}_{n}$,
so that $V^{\bm{l}}_{0}=V_{0}$, $\{V^{\bm{l}}_{n}\}_{n=0}^{\infty}$ is a strictly increasing
sequence of finite subsets of $K^{\bm{l}}$, and $V^{\bm{l}}_{*}$ is dense in $K^{\bm{l}}$.

In particular, for each $l\in\mathbb{N}\setminus\{1,2,3,4\}$ we let
$\bm{l}_{l}:=(l)_{n=1}^{\infty}$ denote the constant sequence with value $l$, set
$K^{l}:=K^{\bm{l}_{l}}$ and $V^{l}_{n}:=V^{\bm{l}_{l}}_{n}$ for $n\in\mathbb{N}\cup\{0\}$,
and call $K^{l}$ the \emph{(two-dimensional) level-$l$ thin Sierpi\'{n}ski gasket}, which
is exactly self-similar in the sense that $K^{l}=\bigcup_{i\in S_{l}}f^{l}_{i}(K^{l})$
(see Figure \ref{fig:thinSGs} and, e.g., \cite[Section 1.1]{Kig01}).
\end{dfn}
%
\begin{figure}[t]\centering
\includegraphics[height=300pt]{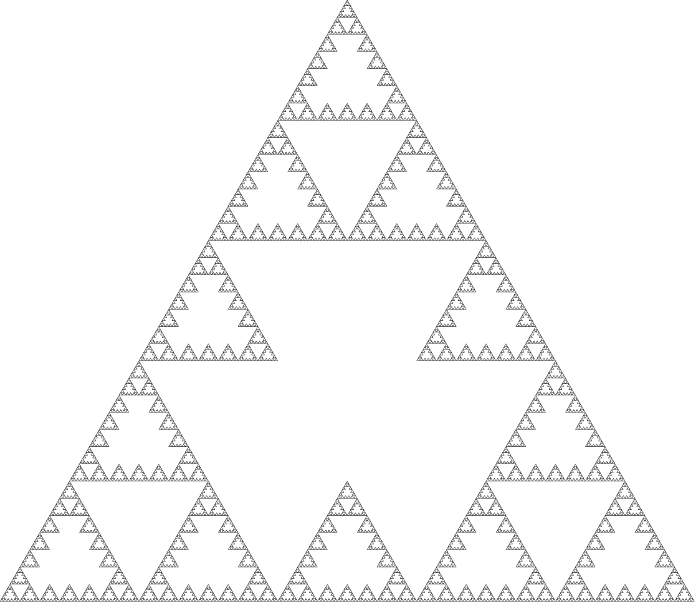}
\caption{A level-$\bm{l}$ thin scale irregular Sierpi\'{n}ski gasket $K^{\bm{l}}$ ($\bm{l}=(5,7,6,12,\ldots)$)}
\label{fig:thinSISG}
\end{figure}

We fix an arbitrary $\bm{l}=(l_{n})_{n=1}^{\infty}\in(\mathbb{N}\setminus\{1,2,3,4\})^{\mathbb{N}}$
in the rest of this section. The following proposition is immediate from Definition \ref{dfn:tsisg}.
\begin{prop}\label{prop:tsisg}
\begin{enumerate}[label=\textup{(\arabic*)},align=left,leftmargin=*]
\item\label{it:tsisg-boundary}Let $w=w_{1}\ldots w_{|w|},v=v_{1}\ldots v_{|v|}\in W^{\bm{l}}_{*}\setminus\{\emptyset\}$
	satisfy $w_{k}\not=v_{k}$ for some $k\in\{1,\ldots,|w|\wedge|v|\}$. Then
	$\#(K^{\bm{l}}_{w}\cap K^{\bm{l}}_{v})\leq 1$ and
	\begin{equation}\label{eq:tsisg-boundary}
	f^{\bm{l}}_{w}(\triangle)\cap f^{\bm{l}}_{v}(\triangle)
		=K^{\bm{l}}_{w}\cap K^{\bm{l}}_{v}
		=F^{\bm{l}}_{w}(V_{0})\cap F^{\bm{l}}_{v}(V_{0}).
	\end{equation}
\item\label{it:tsisg-map-cells}$K^{\bm{l}}=\bigcup_{w\in W^{\bm{l}}_{n}}K^{\bm{l}}_{w}$ for any $n\in\mathbb{N}\cup\{0\}$,
	and $F^{\bm{l}}_{w}(K^{\bm{l}^{|w|}})=K^{\bm{l}}_{w}$ for any $w\in W^{\bm{l}}_{*}$.
\item\label{it:tsisg-map-vertices}$V^{\bm{l}}_{n+k}=\bigcup_{w\in W^{\bm{l}}_{n}}F^{\bm{l}}_{w}(V^{\bm{l}^{n}}_{k})$
	and $V^{\bm{l}}_{*}=\bigcup_{w\in W^{\bm{l}}_{n}}F^{\bm{l}}_{w}(V^{\bm{l}^{n}}_{*})$
	for any $n,k\in\mathbb{N}\cup\{0\}$.
\end{enumerate}
\end{prop}

In exactly the same way as in \cite{Ham92,BH,Ham00} (see also \cite[Part 4]{Kig12}),
we can define a canonical strongly local regular symmetric Dirichlet space
$(K^{\bm{l}},\refmet_{\bm{l}},\refmeas_{\bm{l}},\mathcal{E}^{\bm{l}},\mathcal{F}_{\bm{l}})$
over $K^{\bm{l}}$. First, the metric $\refmet_{\bm{l}}$ on $K^{\bm{l}}$ is defined as follows.
\begin{dfn}\label{dfn:tsisg-metric}
We define $\refmet_{\bm{l}}\colon K^{\bm{l}}\times K^{\bm{l}}\to[0,\infty]$ by
\begin{equation}\label{eq:tsisg-metric}
\refmet_{\bm{l}}(x,y):=\inf\{\ell_{\mathbb{R}^{2}}(\gamma)\mid
	\textrm{$\gamma\colon[0,1]\to K^{\bm{l}}$, $\gamma$ is continuous, $\gamma(0)=x$, $\gamma(1)=y$}\},
\end{equation}
where $\ell_{\mathbb{R}^{2}}(\gamma)$ denotes the Euclidean length of $\gamma$, i.e., the total
variation of the $\mathbb{R}^{2}$-valued map $\gamma$ with respect to the Euclidean norm $|\cdot|$.
We also set $L^{\bm{l}}_{n}:=l_{1}\cdots l_{n}$ ($L^{\bm{l}}_{0}:=1$) for $n\in\mathbb{N}\cup\{0\}$.
\end{dfn}
\begin{prop}\label{prop:tsisg-metric}
$\refmet_{\bm{l}}$ is a metric on $K^{\bm{l}}$, and it is \emph{geodesic},
i.e., for any $x,y\in K^{\bm{l}}$ there exists $\gamma\colon[0,1]\to K^{\bm{l}}$
such that $\gamma(0)=x$, $\gamma(1)=y$ and $\refmet_{\bm{l}}(\gamma(s),\gamma(t))=|s-t|\refmet_{\bm{l}}(x,y)$
for any $s,t\in[0,1]$. Moreover,
\begin{equation}\label{eq:tsisg-metric-geodesic}
|x-y|\leq \refmet_{\bm{l}}(x,y)\leq 6|x-y|\qquad\textrm{for any $x,y\in K^{\bm{l}}$.}
\end{equation}
\end{prop}
\begin{proof}
This proof is similar to \cite[Proof of Lemma 2.4]{BH}, but some additional argument is
required to take care of the possible unboundedness of $\bm{l}=(l_{n})_{n=1}^{\infty}$.
It is immediate from \eqref{eq:tsisg-metric} that $|x-y|\leq \refmet_{\bm{l}}(x,y)<\infty$
for any $x,y\in K^{\bm{l}}$ and thereby that $\refmet_{\bm{l}}$ is a metric on $K^{\bm{l}}$,
which is also geodesic by \cite[Proposition 2.5.19]{BBI}; indeed, the infimum in
\eqref{eq:tsisg-metric} is easily seen to be attained for each $x,y\in K$, by
choosing a sequence $\{\gamma_{n}\}_{n=1}^{\infty}$ of continuous maps as in \eqref{eq:tsisg-metric}
with $\lim_{n\to\infty}\ell_{\mathbb{R}^{2}}(\gamma_{n})=\refmet_{\bm{l}}(x,y)$,
reparameterizing them by arc length on the basis of \cite[Proposition 2.5.9]{BBI},
and applying to them the Arzel\`{a}--Ascoli theorem \cite[Theorem 2.5.14]{BBI}
and the lower semi-continuity \cite[Proposition 2.3.4-(iv)]{BBI} of
$\ell_{\mathbb{R}^{2}}$ with respect to pointwise convergence.

Thus it remains to prove the upper inequality in \eqref{eq:tsisg-metric-geodesic}
for any $x,y\in K^{\bm{l}}$ with $x\not=y$. First, for any $w\in W^{\bm{l}}_{*}$
and any $x\in K^{\bm{l}}_{w}$, we easily see that
\begin{equation}\label{eq:tsisg-metric-diam-cell}
\max_{k\in\{0,1,2\}}\refmet_{\bm{l}}(F^{\bm{l}}_{w}(q_{k}),x)
	\leq\sum_{n=|w|+1}^{\infty}\frac{\frac{3}{2}l_{n}-\frac{5}{2}}{L^{\bm{l}}_{n}}
	\leq\frac{\frac{3}{2}l_{|w|+1}-\frac{5}{2}+\sum_{n=0}^{\infty}\frac{3}{2}(\frac{1}{5})^{n}}{L^{\bm{l}}_{|w|}l_{|w|+1}}
	<\frac{\frac{3}{2}}{L^{\bm{l}}_{|w|}},
\end{equation}
from which it further follows that for any $j,k\in\{0,1,2\}$ with $j\not=k$,
\begin{equation}\label{eq:tsisg-metric-geodesic-proof}
\refmet_{\bm{l}}(F^{\bm{l}}_{w}(q_{k}),x)\leq 5|\langle x-F^{\bm{l}}_{w}(q_{k}),e_{k,j}\rangle|,
\end{equation}
where $e_{k,j}:=q_{j}-q_{k}$. Now let $x,y\in K^{\bm{l}}$ satisfy $x\not=y$ and set
$n_{0}:=\min\{n\in\mathbb{N}\mid\textrm{$\{x,y\}\not\subset K^{\bm{l}}_{w}$ for any $w\in W^{\bm{l}}_{n}$}\}$,
so that $x,y\in K^{\bm{l}}_{w}$ for a unique $w\in W^{\bm{l}}_{n_{0}-1}$ by
Proposition \ref{prop:tsisg}-\ref{it:tsisg-boundary}.
If $K^{\bm{l}}_{wi_{x}}\cap K^{\bm{l}}_{wi_{y}}\not=\emptyset$ for some
$i_{x},i_{y}\in S_{l_{n_{0}}}$ with $x\in K^{\bm{l}}_{wi_{x}}$ and $y\in K^{\bm{l}}_{wi_{y}}$,
then $i_{x}\not=i_{y}$ by the definition of $n_{0}$,
$q_{x,y}=F^{\bm{l}}_{wi_{x}}(q_{k})=F^{\bm{l}}_{wi_{y}}(q_{j})$
for the unique element $q_{x,y}$ of $K^{\bm{l}}_{wi_{x}}\cap K^{\bm{l}}_{wi_{y}}$
and some $j,k\in\{0,1,2\}$ with $j\not=k$ by Proposition \ref{prop:tsisg}-\ref{it:tsisg-boundary},
and from \eqref{eq:tsisg-metric-geodesic-proof} we obtain
\begin{align*}
\refmet_{\bm{l}}(x,y)&\leq \refmet_{\bm{l}}(x,q_{x,y})+\refmet_{\bm{l}}(q_{x,y},y) \\
&\leq 5|\langle x-q_{x,y},e_{k,j}\rangle|+5|\langle q_{x,y}-y,e_{k,j}\rangle|
	=5|\langle x-y,e_{k,j}\rangle|
	\leq 5|x-y|.
\end{align*}
On the other hand, if $K^{\bm{l}}_{wi_{x}}\cap K^{\bm{l}}_{wi_{y}}=\emptyset$ for any
$i_{x},i_{y}\in S_{l_{k}}$ with $x\in K^{\bm{l}}_{wi_{x}}$ and $y\in K^{\bm{l}}_{wi_{y}}$,
then setting
\begin{equation*}
n_{1}:=\min\biggl\{n\in\mathbb{N}\biggm|
	\begin{minipage}{245pt}
	there exists $\{i_{k}\}_{k=0}^{n}\subset S_{l_{n_{0}}}$ such that
	$x\in K^{\bm{l}}_{wi_{0}}$, $y\in K^{\bm{l}}_{wi_{n}}$ and
	$K^{\bm{l}}_{wi_{k-1}}\cap K^{\bm{l}}_{wi_{k}}\not=\emptyset$
	for any $k\in\{1,\ldots,n\}$
	\end{minipage}
	\biggr\},
\end{equation*}
we have $2\leq n_{1}\leq\frac{3}{2}l_{n_{0}}-\frac{5}{2}$,
$L^{\bm{l}}_{n_{0}}\refmet_{\bm{l}}(x,y)\leq\frac{3}{2}+(n_{1}-1)+\frac{3}{2}
	=n_{1}+2\leq\frac{3}{2}l_{n_{0}}$ by \eqref{eq:tsisg-metric-diam-cell},
$\frac{2}{\sqrt{3}}L^{\bm{l}}_{n_{0}}|x-y|\geq(\frac{1}{2}l_{n_{0}}-1)\wedge\lfloor\frac{1}{2}n_{1}\rfloor$,
and thus $\refmet_{\bm{l}}(x,y)/|x-y|\leq\frac{10}{\sqrt{3}}<6$.
\qed\end{proof}

Next, the canonical volume measure $\refmeas_{\bm{l}}$ on $K^{\bm{l}}$ is defined as follows.
\begin{dfn}\label{dfn:tsisg-measure}
We define $\refmeas_{\bm{l}}$ as the unique Borel measure on $K^{\bm{l}}$ such that
\begin{equation}\label{eq:tsisg-measure}
\refmeas_{\bm{l}}(K^{\bm{l}}_{w})=\frac{1}{M^{\bm{l}}_{|w|}}
	\qquad\textrm{for any $w\in W^{\bm{l}}_{*}$,}
\end{equation}
where $M^{\bm{l}}_{n}:=(\# S_{l_{1}})\cdots(\# S_{l_{n}})=\prod_{k=1}^{n}(3l_{k}-3)$
($M^{\bm{l}}_{0}:=1$) for $n\in\mathbb{N}\cup\{0\}$.
\end{dfn}
The measure $\refmeas_{\bm{l}}$ can be considered as the ``uniform distribution on $K^{\bm{l}}$''.
Its uniqueness stated in Definition \ref{dfn:tsisg-measure} is immediate from the
Dynkin class theorem (see, e.g., \cite[Appendixes, Theorem 4.2]{EK}).
It is also easily seen to be obtained as
$\refmeas_{\bm{l}}=\bigl(\prod_{n=1}^{\infty}\unif(S_{l_{n}})\bigr)(\pi_{\bm{l}}^{-1}(\cdot))$,
where $\unif(S_{l_{n}})$ denotes the uniform distribution on $S_{l_{n}}$,
$\prod_{n=1}^{\infty}\unif(S_{l_{n}})$ their product probability measure on $\prod_{n=1}^{\infty}S_{l_{n}}$
(see, e.g., \cite[Theorem 8.2.2]{Dud} for its unique existence) and
$\pi_{\bm{l}}\colon\prod_{n=1}^{\infty}S_{l_{n}}\to K^{\bm{l}}$ the continuous surjection given by
$\{\pi_{\bm{l}}((\omega_{n})_{n=1}^{\infty})\}:=\bigcap_{n=1}^{\infty}K^{\bm{l}}_{\omega_{1}\ldots\omega_{n}}$.

Now we turn to the construction of the canonical Dirichlet form (resistance form)
$(\mathcal{E}^{\bm{l}},\mathcal{F}_{\bm{l}})$ on $K^{\bm{l}}$, which is achieved by
taking the ``inductive limit'' of a certain canonical sequence of discrete Dirichlet
forms on the finite sets $\{V^{\bm{l}}_{n}\}_{n=0}^{\infty}$ via the standard
method presented in \cite[Chapters 2 and 3]{Kig01} (see also \cite[Sections 6 and 7]{Bar98}).
The whole construction is based on the following definition and lemma.%
\begin{dfn}\label{dfn:tsisg-form-0}
Recalling that $V^{\bm{l}}_{0}=V_{0}$, we define a non-negative definite symmetric bilinear form
$\mathcal{E}^{0}\colon\mathbb{R}^{V_{0}}\times\mathbb{R}^{V_{0}}\to\mathbb{R}$
on $\mathbb{R}^{V_{0}}=\mathbb{R}^{V^{\bm{l}}_{0}}$ by
\begin{equation}\label{eq:tsisg-form-0}
\mathcal{E}^{0}(u,v)
	:=\frac{1}{2}\sum_{j,k=0}^{2}(u(q_{j})-u(q_{k}))(v(q_{j})-v(q_{k})),
	\qquad u,v\in\mathbb{R}^{V_{0}},
\end{equation}
and set $r_{l}:=(\frac{2}{3}l+\frac{1}{9})^{-1}$ for each $l\in\mathbb{N}\setminus\{1,2,3,4\}$.
\end{dfn}
The value of $r_{l}$ is specifically chosen in order for the following lemma to hold.
\begin{lem}\label{lem:tsisg-rl}
Let $l\in\mathbb{N}\setminus\{1,2,3,4\}$. Then for any $u\in\mathbb{R}^{V_{0}}$,
\begin{equation}\label{eq:tsisg-rl}
\min\Biggl\{\sum_{i\in S_{l}}\mathcal{E}^{0}\bigl(v\circ F^{l}_{i}|_{V_{0}},v\circ F^{l}_{i}|_{V_{0}}\bigr)
	\Biggm|\textrm{$v\in\mathbb{R}^{V^{l}_{1}}$, $v|_{V_{0}}=u$}\Biggr\}
	=r_{l}\mathcal{E}^{0}(u,u).
\end{equation}
\end{lem}
\begin{proof}
This is immediate from a direct calculation using the $\mathrm{\Delta}$--Y transform
(see, e.g., \cite[Lemma 2.1.15]{Kig01}).
\qed\end{proof}

We would like to define a bilinear form $\mathcal{E}^{\bm{l},n}$ on
$\mathbb{R}^{V^{\bm{l}}_{n}}$ for each $n\in\mathbb{N}$ as the sum of the copies of
\eqref{eq:tsisg-form-0} on $\{F^{\bm{l}}_{w}(V_{0})\}_{w\in W^{\bm{l}}_{n}}$ and then
to take their limit as $n\to\infty$, which is enabled by introducing the scaling factors
$R^{\bm{l}}_{n}$ suggested by Lemma \ref{lem:tsisg-rl} as in the following definition.
\begin{dfn}\label{dfn:tsisg-form-ln}
For each $n\in\mathbb{N}\cup\{0\}$, we define a non-negative definite symmetric bilinear form
$\mathcal{E}^{\bm{l},n}\colon\mathbb{R}^{V^{\bm{l}}_{n}}\times\mathbb{R}^{V^{\bm{l}}_{n}}\to\mathbb{R}$
on $\mathbb{R}^{V^{\bm{l}}_{n}}$ by
\begin{equation}\label{eq:tsisg-form-ln}
\mathcal{E}^{\bm{l},n}(u,v)
	:=\frac{1}{R^{\bm{l}}_{n}}\sum_{w\in W^{\bm{l}}_{n}}\mathcal{E}^{0}\bigl(u\circ F^{\bm{l}}_{w}|_{V_{0}},v\circ F^{\bm{l}}_{w}|_{V_{0}}\bigr),
	\qquad u,v\in\mathbb{R}^{V^{\bm{l}}_{n}},
\end{equation}
where $R^{\bm{l}}_{n}:=r_{l_{1}}\cdots r_{l_{n}}=\prod_{k=1}^{n}(\frac{2}{3}l_{k}+\frac{1}{9})^{-1}$
($R^{\bm{l}}_{0}:=1$), so that $\mathcal{E}^{\bm{l},0}=\mathcal{E}^{0}$.
\end{dfn}
\begin{prop}\label{prop:tsisg-form-compatible}
The sequence $\{\mathcal{E}^{\bm{l},n}\}_{n=0}^{\infty}$ of forms is \emph{compatible}, i.e.,
for any $n,k\in\mathbb{N}\cup\{0\}$ and any $u\in\mathbb{R}^{V^{\bm{l}}_{n}}$,
\begin{equation}\label{eq:tsisg-form-compatible}
\min\bigl\{\mathcal{E}^{\bm{l},n+k}(v,v)\bigm|\textrm{$v\in\mathbb{R}^{V^{\bm{l}}_{n+k}}$, $v|_{V^{\bm{l}}_{n}}=u$}\bigr\}
	=\mathcal{E}^{\bm{l},n}(u,u).
\end{equation}
\end{prop}
\begin{proof}
This is immediate from an induction on $k$ based on Lemma \ref{lem:tsisg-rl}.
\qed\end{proof}
Proposition \ref{prop:tsisg-form-compatible} allows us to take the ``inductive limit'' of
$\{\mathcal{E}^{\bm{l},n}\}_{n=0}^{\infty}$ as in the following definition. Note that
$\{\mathcal{E}^{\bm{l},n}(u|_{V^{\bm{l}}_{n}},u|_{V^{\bm{l}}_{n}})\}_{n=0}^{\infty}\subset[0,\infty)$
is non-decreasing by \eqref{eq:tsisg-form-compatible} and hence has a limit in $[0,\infty]$
for any $u\in\mathbb{R}^{V^{\bm{l}}_{*}}$.
\begin{dfn}\label{dfn:tsisg-RF-Vstar}
We define a linear subspace $\mathcal{F}_{\bm{l}}$ of $\mathbb{R}^{V^{\bm{l}}_{*}}$
and a non-negative definite symmetric bilinear form
$\mathcal{E}^{\bm{l}}\colon\mathcal{F}_{\bm{l}}\times\mathcal{F}_{\bm{l}}\to\mathbb{R}$
on $\mathcal{F}_{\bm{l}}$ by
\begin{align}\label{eq:tsisg-RF-domain-Vstar}
\mathcal{F}_{\bm{l}}&:=\Bigl\{u\in\mathbb{R}^{V^{\bm{l}}_{*}}\Bigm|\lim_{n\to\infty}\mathcal{E}^{\bm{l},n}(u|_{V^{\bm{l}}_{n}},u|_{V^{\bm{l}}_{n}})<\infty\Bigr\},\\
\mathcal{E}^{\bm{l}}(u,v)&:=\lim_{n\to\infty}\mathcal{E}^{\bm{l},n}(u|_{V^{\bm{l}}_{n}},v|_{V^{\bm{l}}_{n}})\in\mathbb{R},
	\quad u,v\in\mathcal{F}_{\bm{l}}.
\label{eq:tsisg-RF-form}
\end{align}
\end{dfn}
Then applying \cite[Lemma 2.2.2, Proposition 2.2.4, Lemma 2.2.5 and Theorem 2.2.6]{Kig01}
on the basis of Proposition \ref{prop:tsisg-form-compatible}, we obtain the following proposition.
See \cite[Definition 2.3.1]{Kig01} or \cite[Definition 3.1]{Kig12} for the notion of resistance forms.%
\begin{prop}\label{prop:tsisg-RF}
$(\mathcal{E}^{\bm{l}},\mathcal{F}_{\bm{l}})$ is a \emph{resistance form} on $V^{\bm{l}}_{*}$,
i.e., the following hold:
\begin{enumerate}[label=\textup{(RF\arabic*)},align=left,leftmargin=*]
\item\label{it:RF1}$\{u\in\mathcal{F}_{\bm{l}}\mid\mathcal{E}^{\bm{l}}(u,u)=0\}=\mathbb{R}\one_{V^{\bm{l}}_{*}}$.
\item\label{it:RF2}$(\mathcal{F}_{\bm{l}}/\mathbb{R}\one_{V^{\bm{l}}_{*}},\mathcal{E}^{\bm{l}})$ is a Hilbert space.
\item\label{it:RF3}$\{u|_{V}\mid u\in\mathcal{F}_{\bm{l}}\}=\mathbb{R}^{V}$ for any non-empty finite subset $V$ of $V^{\bm{l}}_{*}$.
\item\label{it:RF4}$R_{\mathcal{E}^{\bm{l}}}(x,y):=\sup\biggl\{\dfrac{|u(x)-u(y)|^{2}}{\mathcal{E}^{\bm{l}}(u,u)}\biggm|u\in\mathcal{F}_{\bm{l}}\setminus\mathbb{R}\one_{V^{\bm{l}}_{*}}\biggr\}<\infty$
	for any $x,y\in V^{\bm{l}}_{*}$.
\item\label{it:RF5}$u^{+}\wedge 1\in\mathcal{F}_{\bm{l}}$ and
	$\mathcal{E}^{\bm{l}}(u^{+}\wedge 1,u^{+}\wedge 1)\leq\mathcal{E}^{\bm{l}}(u,u)$
	for any $u\in\mathcal{F}_{\bm{l}}$.
\end{enumerate}
Moreover, $R_{\mathcal{E}^{\bm{l}}}\colon V^{\bm{l}}_{*}\times V^{\bm{l}}_{*}\to[0,\infty)$
is a metric on $V^{\bm{l}}_{*}$, called the \emph{resistance metric} of
$(\mathcal{E}^{\bm{l}},\mathcal{F}_{\bm{l}})$, and
for any $u\in\mathcal{F}_{\bm{l}}$ and any $x,y\in V^{\bm{l}}_{*}$,
\begin{equation}\label{eq:tsisg-RF-Hoelder}
|u(x)-u(y)|^{2}\leq R_{\mathcal{E}^{\bm{l}}}(x,y)\mathcal{E}^{\bm{l}}(u,u).
\end{equation}
\end{prop}

Recalling Proposition \ref{prop:tsisg}-\ref{it:tsisg-map-vertices}, we also see
from the above construction that the following (non-exact) self-similarity of
$(\mathcal{E}^{\bm{l}},\mathcal{F}_{\bm{l}})$ holds.
\begin{prop}\label{prop:tsisg-RF-self-sim}
Let $n\in\mathbb{N}\cup\{0\}$. Then
\begin{align}\label{eq:tsisg-RF-self-sim-domain}
\mathcal{F}_{\bm{l}}&=\bigl\{u\in\mathbb{R}^{V^{\bm{l}}_{*}}\bigm|\textrm{$u\circ F^{\bm{l}}_{w}|_{V^{\bm{l}^{n}}_{*}}\in\mathcal{F}_{\bm{l}^{n}}$ for any $w\in W^{\bm{l}}_{n}$}\bigr\},\\
\mathcal{E}^{\bm{l}}(u,v)&=\frac{1}{R^{\bm{l}}_{n}}\sum_{w\in W^{\bm{l}}_{n}}\mathcal{E}^{\bm{l}^{n}}\bigl(u\circ F^{\bm{l}}_{w}|_{V^{\bm{l}^{n}}_{*}},v\circ F^{\bm{l}}_{w}|_{V^{\bm{l}^{n}}_{*}}\bigr)
	\quad\textrm{for any $u,v\in\mathcal{F}_{\bm{l}}$.}
\label{eq:tsisg-RF-self-sim-form}
\end{align}
\end{prop}
\begin{proof}
It follows from Proposition \ref{prop:tsisg}-\ref{it:tsisg-map-vertices} and
\eqref{eq:tsisg-form-ln} that for each $n,k\in\mathbb{N}\cup\{0\}$,
\begin{equation*}
\mathcal{E}^{\bm{l},n+k}(u,v)
	=\frac{1}{R^{\bm{l}}_{n}}\sum_{w\in W^{\bm{l}}_{n}}\mathcal{E}^{\bm{l}^{n},k}\bigl(u\circ F^{\bm{l}}_{w}|_{V^{\bm{l}^{n}}_{k}},v\circ F^{\bm{l}}_{w}|_{V^{\bm{l}^{n}}_{k}}\bigr)
	\quad\textrm{for any $u,v\in\mathbb{R}^{V^{\bm{l}}_{n+k}}$,}
\end{equation*}
which together with \eqref{eq:tsisg-RF-domain-Vstar} and \eqref{eq:tsisg-RF-form}
immediately yields \eqref{eq:tsisg-RF-self-sim-domain} and \eqref{eq:tsisg-RF-self-sim-form}.
\qed\end{proof}
\begin{lem}\label{lem:tsisg-RM-contraction}
For any $w\in W^{\bm{l}}_{*}$ and any $x,y\in V^{\bm{l}^{|w|}}_{*}$,
\begin{equation}\label{eq:tsisg-RM-contraction}
R_{\mathcal{E}^{\bm{l}}}(F^{\bm{l}}_{w}(x),F^{\bm{l}}_{w}(y))
	\leq R^{\bm{l}}_{|w|}R_{\mathcal{E}^{\bm{l}^{|w|}}}(x,y).
\end{equation}
\end{lem}
\begin{proof}
This is immediate from Proposition \ref{prop:tsisg-RF-self-sim} and
Proposition \ref{prop:tsisg-RF}-\ref{it:RF4}.
\qed\end{proof}

Later we will use the following definition and proposition several times.
\begin{dfn}\label{dfn:tsisg-harmonic-Vn}
Let $h\in\mathbb{R}^{V^{\bm{l}}_{*}}$ and $n\in\mathbb{N}\cup\{0\}$.
We say that $h$ is \emph{$\mathcal{E}^{\bm{l}}$-harmonic} off $V^{\bm{l}}_{n}$
if and only if $h\in\mathcal{F}_{\bm{l}}$ and
\begin{equation}\label{eq:tsisg-harmonic-Vn}
\begin{split}
\mathcal{E}^{\bm{l}}(h,h)&=\min\{\mathcal{E}^{\bm{l}}(v,v)\mid\textrm{$v\in\mathcal{F}_{\bm{l}}$, $v|_{V^{\bm{l}}_{n}}=h|_{V^{\bm{l}}_{n}}$}\}, \\
\textrm{or equivalently,}\quad
	\mathcal{E}^{\bm{l}}(h,v)&=0
	\quad\textrm{for any $v\in\mathcal{F}_{\bm{l}}$ with $v|_{V^{\bm{l}}_{n}}=0$.}
\end{split}
\end{equation}
We set
$\mathcal{H}_{\bm{l},n}:=\{h\in\mathbb{R}^{V^{\bm{l}}_{*}}\mid
	\textrm{$h$ is $\mathcal{E}^{\bm{l}}$-harmonic off $V^{\bm{l}}_{n}$}\}$,
which is a linear subspace of $\mathcal{F}_{\bm{l}}$.
\end{dfn}
\begin{prop}\label{prop:tsisg-harmonic-Vn}
Let $n\in\mathbb{N}\cup\{0\}$. Then for each $h\in\mathbb{R}^{V^{\bm{l}}_{*}}$, the
following four conditions \ref{it:tsisg-harmonic-Vn}, \ref{it:tsisg-harmonic-Vn-linear-eq},
\ref{it:tsisg-harmonic-Vn-trace} and \ref{it:tsisg-harmonic-Vn-self-sim} are equivalent
to each other and imply \ref{it:tsisg-harmonic-Vn-maximum-principle} below:
\begin{enumerate}[label=\textup{(\arabic*)},align=left,leftmargin=*]
\item\label{it:tsisg-harmonic-Vn}$h\in\mathcal{H}_{\bm{l},n}$.
\item\label{it:tsisg-harmonic-Vn-linear-eq}$\sum_{y\in V^{\bm{l}}_{n+k},\,L^{\bm{l}}_{n+k}\refmet_{\bm{l}}(x,y)=1}(h(y)-h(x))=0$
	for any $k\in\mathbb{N}$ and any $x\in V^{\bm{l}}_{n+k}\setminus V^{\bm{l}}_{n}$.
\item\label{it:tsisg-harmonic-Vn-trace}$h\in\mathcal{F}_{\bm{l}}$ and
	$\mathcal{E}^{\bm{l}}(h,h)=\mathcal{E}^{\bm{l},n}(h|_{V^{\bm{l}}_{n}},h|_{V^{\bm{l}}_{n}})$.
\item\label{it:tsisg-harmonic-Vn-self-sim}$h\circ F^{\bm{l}}_{w}|_{V^{\bm{l}^{n}}_{*}}\in\mathcal{H}_{\bm{l}^{n},0}$
	for any $w\in W^{\bm{l}}_{n}$.
\item\label{it:tsisg-harmonic-Vn-maximum-principle}\textup{(Maximum principle)}
	For any $w\in W^{\bm{l}}_{n}$ and any $x\in F^{\bm{l}}_{w}(V^{\bm{l}^{n}}_{*})$,
	\begin{equation}\label{eq:tsisg-harmonic-Vn-maximum-principle}
	\min_{q\in F^{\bm{l}}_{w}(V_{0})}h(q)\leq h(x)\leq\max_{q\in F^{\bm{l}}_{w}(V_{0})}h(q).
	\end{equation}
\end{enumerate}
Also, for each $u\in\mathbb{R}^{V^{\bm{l}}_{n}}$ there exists a unique
$h^{\bm{l}}_{n}(u)\in\mathcal{H}_{\bm{l},n}$ with $h^{\bm{l}}_{n}(u)|_{V^{\bm{l}}_{n}}=u$,
and the map $h^{\bm{l}}_{n}\colon\mathbb{R}^{V^{\bm{l}}_{n}}\to\mathcal{H}_{\bm{l},n}$
is a linear isomorphism.
\end{prop}
\begin{proof}
The assertions for $h^{\bm{l}}_{n}$ and the equivalence of \ref{it:tsisg-harmonic-Vn},
\ref{it:tsisg-harmonic-Vn-linear-eq} and \ref{it:tsisg-harmonic-Vn-trace} follow from
Proposition \ref{prop:tsisg-form-compatible}, \cite[Lemma 2.2.2]{Kig01} and \eqref{eq:tsisg-form-ln}.
Moreover, noting that $\mathcal{E}^{\bm{l}^{n}}(u,u)\geq\mathcal{E}^{0}(u|_{V_{0}},u|_{V_{0}})$
for any $u\in\mathcal{F}_{\bm{l}^{n}}$, we easily see from \eqref{eq:tsisg-RF-self-sim-domain},
\eqref{eq:tsisg-RF-self-sim-form} and \eqref{eq:tsisg-form-ln} that
\ref{it:tsisg-harmonic-Vn-trace} holds if and only if
$h\circ F^{\bm{l}}_{w}|_{V^{\bm{l}^{n}}_{*}}\in\mathcal{F}_{\bm{l}^{n}}$ and
$\mathcal{E}^{\bm{l}^{n}}\bigl(h\circ F^{\bm{l}}_{w}|_{V^{\bm{l}^{n}}_{*}},h\circ F^{\bm{l}}_{w}|_{V^{\bm{l}^{n}}_{*}}\bigr)
	=\mathcal{E}^{0}\bigl(h\circ F^{\bm{l}}_{w}|_{V_{0}},h\circ F^{\bm{l}}_{w}|_{V_{0}}\bigr)$
for any $w\in W^{\bm{l}}_{n}$, which is equivalent to \ref{it:tsisg-harmonic-Vn-self-sim}
by the equivalence of \ref{it:tsisg-harmonic-Vn-trace} and \ref{it:tsisg-harmonic-Vn}
with $h\circ F^{\bm{l}}_{w}|_{V^{\bm{l}^{n}}_{*}},\bm{l}^{n},0$ in place of $h,\bm{l},n$.
Lastly, \ref{it:tsisg-harmonic-Vn-self-sim} implies \ref{it:tsisg-harmonic-Vn-maximum-principle}
by \cite[Lemma 2.2.3]{Kig01} applied to $h\circ F^{\bm{l}}_{w}|_{V^{\bm{l}^{n}}_{*}}$
for each $w\in W^{\bm{l}}_{n}$.
\qed\end{proof}

Note that at this stage the domain $\mathcal{F}_{\bm{l}}$ of $\mathcal{E}^{\bm{l}}$
is only a linear subspace of $\mathbb{R}^{V^{\bm{l}}_{*}}$, unlike that of a regular
symmetric Dirichlet form on $L^{2}(K^{\bm{l}},\refmeas_{\bm{l}})$, which is a linear
subspace of $L^{2}(K^{\bm{l}},\refmeas_{\bm{l}})$ including a dense subalgebra of
$(\mathcal{C}(K^{\bm{l}}),\|\cdot\|_{\sup})$. As the last step of the construction
of the canonical Dirichlet form on $K^{\bm{l}}$, we now fill this gap by proving that
$\id_{V^{\bm{l}}_{*}}\colon(V^{\bm{l}}_{*},\refmet_{\bm{l}}|_{V^{\bm{l}}_{*}\times V^{\bm{l}}_{*}})
	\to(V^{\bm{l}}_{*},R_{\mathcal{E}^{\bm{l}}})$
is uniformly continuous with uniformly continuous inverse and consequently that each
$u\in\mathcal{F}_{\bm{l}}$ uniquely extends to an element of $\mathcal{C}(K^{\bm{l}})$
by virtue of \eqref{eq:tsisg-RF-Hoelder}.
\begin{prop}\label{prop:tsisg-RM-completion}
For any $x,y\in V^{\bm{l}}_{*}$ and any $n\in\mathbb{N}$, the following hold:
\begin{enumerate}[label=\textup{(\arabic*)},align=left,leftmargin=*]
\item\label{it:tsisg-unifcont-Euc-RM}If $\refmet_{\bm{l}}(x,y)<1/L^{\bm{l}}_{n}$,
	then $R_{\mathcal{E}^{\bm{l}}}(x,y)\leq 4R^{\bm{l}}_{n}$.
\item\label{it:tsisg-unifcont-RM-Euc}If $R_{\mathcal{E}^{\bm{l}}}(x,y)<\frac{1}{6}R^{\bm{l}}_{n}$,
	then $\refmet_{\bm{l}}(x,y)\leq 3/L^{\bm{l}}_{n}$.
\end{enumerate}
In particular, $R_{\mathcal{E}^{\bm{l}}}$ uniquely extends to
$\overline{R}_{\mathcal{E}^{\bm{l}}}\in\mathcal{C}(K^{\bm{l}}\times K^{\bm{l}})$,
$\overline{R}_{\mathcal{E}^{\bm{l}}}$ is a metric on $K^{\bm{l}}$ compatible with
the original (Euclidean) topology of $K^{\bm{l}}$, and
$\bigl((K^{\bm{l}},\overline{R}_{\mathcal{E}^{\bm{l}}}),\id_{V^{\bm{l}}_{*}}\bigr)$
is the completion of $(V^{\bm{l}}_{*},R_{\mathcal{E}^{\bm{l}}})$.
\end{prop}
\begin{proof}
We essentially follow \cite[Chapter 22]{Kig12}, but the possible unboundedness of
$\bm{l}=(l_{n})_{n=1}^{\infty}$ requires some additional care. First, since
\begin{equation*}
R_{\mathcal{E}^{\bm{l}}}(q_{j},q_{k})
	=\bigl(\min\{\mathcal{E}^{0}(u,u)\mid\textrm{$u\in\mathbb{R}^{V_{0}}$, $u(q_{j})=1$, $u(q_{k})=0$}\}\bigr)^{-1}
	=\frac{2}{3}
\end{equation*}
for any $j,k\in\{0,1,2\}$ with $j\not=k$ by \cite[(2.2.3) and Lemma 2.2.5]{Kig01},
it follows from Lemma \ref{lem:tsisg-RM-contraction} that
for any $w\in W^{\bm{l}}_{*}$ and any $j,k\in\{0,1,2\}$,
\begin{equation}\label{eq:tsisg-RM-contraction-cell-vertices}
R_{\mathcal{E}^{\bm{l}}}(F^{\bm{l}}_{w}(q_{j}),F^{\bm{l}}_{w}(q_{k}))
	\leq\frac{2}{3}R^{\bm{l}}_{|w|}.
\end{equation}
Recalling that $R_{\mathcal{E}^{\bm{l}}}$ is a metric on $V^{\bm{l}}_{*}$ as stated in
Proposition \ref{prop:tsisg-RF}, we easily see from \eqref{eq:tsisg-RM-contraction-cell-vertices}
and the triangle inequality for $R_{\mathcal{E}^{\bm{l}}}$ that for any $x\in V^{\bm{l}}_{*}$,
\begin{equation}\label{eq:tsisg-RM-diam}
\max_{k\in\{0,1,2\}}R_{\mathcal{E}^{\bm{l}}}(q_{k},x)
	\leq\sum_{n=1}^{\infty}\frac{3l_{n}-5}{2}\cdot\frac{2}{3}R^{\bm{l}}_{n}
	\leq\frac{l_{1}-\frac{5}{3}+\sum_{n=2}^{\infty}\frac{3}{2}(\frac{9}{31})^{n-2}}{\frac{2}{3}l_{1}+\frac{1}{9}}
	<2,
\end{equation}
which together with Lemma \ref{lem:tsisg-RM-contraction} further implies that
for any $w\in W^{\bm{l}}_{*}$ and any $x\in F^{\bm{l}}_{w}(V^{\bm{l}^{|w|}}_{*})$,
\begin{equation}\label{eq:tsisg-RM-diam-cell}
\max_{k\in\{0,1,2\}}R_{\mathcal{E}^{\bm{l}}}(F^{\bm{l}}_{w}(q_{k}),x)<2R^{\bm{l}}_{|w|}.
\end{equation}

To see \ref{it:tsisg-unifcont-Euc-RM} and \ref{it:tsisg-unifcont-RM-Euc},
let $x,y\in V^{\bm{l}}_{*}$, $n\in\mathbb{N}$,
choose $w\in W^{\bm{l}}_{n}$ so that $x\in K^{\bm{l}}_{w}$, and set
$\Lambda_{n,w}:=\{v\in W^{\bm{l}}_{n}\mid K^{\bm{l}}_{w}\cap K^{\bm{l}}_{v}\not=\emptyset\}$
and $U_{n,w}:=\bigcup_{v\in\Lambda_{n,w}}K^{\bm{l}}_{v}$. It holds that
\begin{equation}\label{eq:tsisg-unifcont-cell}
\textrm{if $y\in U_{n,w}$, then}\qquad
	\refmet_{\bm{l}}(x,y)<\frac{3}{L^{\bm{l}}_{n}}
	\qquad\textrm{and}\qquad
	R_{\mathcal{E}^{\bm{l}}}(x,y)<4R^{\bm{l}}_{n}
\end{equation}
by \eqref{eq:tsisg-boundary}, the triangle inequality for $\refmet_{\bm{l}}$ and
$R_{\mathcal{E}^{\bm{l}}}$, \eqref{eq:tsisg-metric-diam-cell} and \eqref{eq:tsisg-RM-diam-cell}.
On the other hand, if $y\not\in U_{n,w}$, then clearly $\refmet_{\bm{l}}(x,y)\geq 1/L^{\bm{l}}_{n}$
by \eqref{eq:tsisg-boundary} and \eqref{eq:tsisg-metric}, and
recalling Proposition \ref{prop:tsisg-harmonic-Vn} and setting
$h_{n,w}:=h^{\bm{l}}_{n}(\one_{F^{\bm{l}}_{w}(V_{0})})$, we have
$h_{n,w}|_{F^{\bm{l}}_{w}(V^{\bm{l}^{n}}_{*})}=1$,
$h_{n,w}|_{F^{\bm{l}}_{v}(V^{\bm{l}^{n}}_{*})}=0$ for any
$v\in W^{\bm{l}}_{n}\setminus\Lambda_{n,w}$,
$\mathcal{E}^{\bm{l}}(h_{n,w},h_{n,w})
	=\mathcal{E}^{\bm{l},n}(\one_{F^{\bm{l}}_{w}(V_{0})},\one_{F^{\bm{l}}_{w}(V_{0})})$,
and therefore
\begin{equation*}
R_{\mathcal{E}^{\bm{l}}}(x,y)
	\geq\frac{|h_{n,w}(x)-h_{n,w}(y)|^{2}}{\mathcal{E}^{\bm{l}}(h_{n,w},h_{n,w})}
	=\frac{1}{\mathcal{E}^{\bm{l},n}(\one_{F^{\bm{l}}_{w}(V_{0})},\one_{F^{\bm{l}}_{w}(V_{0})})}
	=\frac{\frac{1}{2}R^{\bm{l}}_{n}}{\#\Lambda_{n,w}-1}
	\geq\frac{R^{\bm{l}}_{n}}{6}
\end{equation*}
by Proposition \ref{prop:tsisg-RF}-\ref{it:RF4}, \eqref{eq:tsisg-form-ln},
\eqref{eq:tsisg-form-0}, Proposition \ref{prop:tsisg}-\ref{it:tsisg-boundary} and
$\#\Lambda_{n,w}\leq 4$. It follows that, if either $\refmet_{\bm{l}}(x,y)<1/L^{\bm{l}}_{n}$
or $R_{\mathcal{E}^{\bm{l}}}(x,y)<\frac{1}{6}R^{\bm{l}}_{n}$, then
$y\in U_{n,w}$, hence $\refmet_{\bm{l}}(x,y)<3/L^{\bm{l}}_{n}$ and
$R_{\mathcal{E}^{\bm{l}}}(x,y)<4R^{\bm{l}}_{n}$ by \eqref{eq:tsisg-unifcont-cell},
proving \ref{it:tsisg-unifcont-Euc-RM} and \ref{it:tsisg-unifcont-RM-Euc}, which in turn
immediately imply the existence and the stated properties of $\overline{R}_{\mathcal{E}^{\bm{l}}}$.
\qed\end{proof}
\begin{dfn}\label{dfn:tsisg-RF}
Throughout the rest of this paper, we identify $\mathcal{F}_{\bm{l}}$
with the linear subspace of $\mathcal{C}(K^{\bm{l}})$ given by
\begin{equation}\label{eq:tsisg-RF-domain}
\{u\in\mathcal{C}(K^{\bm{l}})\mid u|_{V^{\bm{l}}_{*}}\in\mathcal{F}_{\bm{l}}\}
	=\Bigl\{u\in\mathcal{C}(K^{\bm{l}})\Bigm|\lim_{n\to\infty}\mathcal{E}^{\bm{l},n}(u|_{V^{\bm{l}}_{n}},u|_{V^{\bm{l}}_{n}})<\infty\Bigr\}
\end{equation}
through the mapping $u\mapsto u|_{V^{\bm{l}}_{*}}$, which is a linear isomorphism from
\eqref{eq:tsisg-RF-domain} to $\mathcal{F}_{\bm{l}}$ since each $u\in\mathcal{F}_{\bm{l}}$
uniquely extends to an element of $\mathcal{C}(K^{\bm{l}})$
by Proposition \ref{prop:tsisg-RM-completion} and \eqref{eq:tsisg-RF-Hoelder}.
The pair $(\mathcal{E}^{\bm{l}},\mathcal{F}_{\bm{l}})$ is then called
the \emph{canonical resistance form} on $K^{\bm{l}}$.
\end{dfn}
\begin{thm}\label{thm:tsisg-RF}
\begin{enumerate}[label=\textup{(\arabic*)},align=left,leftmargin=*]
\item\label{it:tsisg-RF}$(\mathcal{E}^{\bm{l}},\mathcal{F}_{\bm{l}})$ is a resistance
	form on $K^{\bm{l}}$ with resistance metric $\overline{R}_{\mathcal{E}^{\bm{l}}}$,
	which is hereafter denoted as $R_{\mathcal{E}^{\bm{l}}}$ for simplicity of the notation.
\item\label{it:tsisg-RF-regular}$(\mathcal{E}^{\bm{l}},\mathcal{F}_{\bm{l}})$ is \emph{regular}, i.e.,
	$\mathcal{F}_{\bm{l}}$ is a dense subalgebra of $(\mathcal{C}(K^{\bm{l}}),\|\cdot\|_{\sup})$.
\item\label{it:tsisg-RF-strongly-local}$(\mathcal{E}^{\bm{l}},\mathcal{F}_{\bm{l}})$ is
	\emph{strongly local}, i.e., $\mathcal{E}^{\bm{l}}(u,v)=0$ for any $u,v\in\mathcal{F}_{\bm{l}}$
	that satisfy $\supp_{K}[u-a\one_{K^{\bm{l}}}]\cap\supp_{K}[v]=\emptyset$ for some $a\in\mathbb{R}$.
\end{enumerate}
\end{thm}
\begin{proof}
\ref{it:tsisg-RF} follows from Propositions \ref{prop:tsisg-RF}, \ref{prop:tsisg-RM-completion},
Definition \ref{dfn:tsisg-RF}, \cite[Lemma 2.3.9 and Theorem 2.3.10]{Kig01},
\ref{it:tsisg-RF-regular} from \ref{it:tsisg-RF}, the compactness of
$(K^{\bm{l}},R_{\mathcal{E}^{\bm{l}}})$, \cite[Corollary 6.4 and Lemma 6.5]{Kig12},
and \ref{it:tsisg-RF-strongly-local} from $\one_{K^{\bm{l}}}\in\mathcal{F}_{\bm{l}}$,
$\mathcal{E}^{\bm{l}}(\one_{K^{\bm{l}}},\one_{K^{\bm{l}}})=0$ and \eqref{eq:tsisg-RF-self-sim-form}.
\qed\end{proof}
\begin{rmk}\label{rmk:tsisg-RF}
\begin{enumerate}[label=\textup{(\arabic*)},align=left,leftmargin=*]
\item\label{it:thm:tsisg-RF}To be explicit, Theorem \ref{thm:tsisg-RF}-\ref{it:tsisg-RF} means the following:
	
	\emph{Proposition \textup{\ref{prop:tsisg-RF}}-\ref{it:RF1},\ref{it:RF2},\ref{it:RF3},\ref{it:RF5}
	with $K^{\bm{l}}$ in place of $V^{\bm{l}}_{*}$ hold and
	$\overline{R}_{\mathcal{E}^{\bm{l}}}(x,y)=\sup\bigl\{|u(x)-u(y)|^{2}/\mathcal{E}^{\bm{l}}(u,u)
		\bigm|u\in\mathcal{F}_{\bm{l}}\setminus\mathbb{R}\one_{K^{\bm{l}}}\bigr\}$
	for any $x,y\in K^{\bm{l}}$.}
\item\label{it:tsisg-RF-drop-Vstar}Under the conventions introduced in
	Definition \ref{dfn:tsisg-RF} and Theorem \ref{thm:tsisg-RF}-\ref{it:tsisg-RF},
	we easily get the following, \emph{which we will utilize below without further notice}:
	\begin{itemize}[align=left,leftmargin=*]
	\item\eqref{eq:tsisg-RF-Hoelder} for any $u\in\mathcal{F}_{\bm{l}}$ and any $x,y\in K^{\bm{l}}$;
	\item Proposition \ref{prop:tsisg-RF-self-sim} with $\mathcal{C}(K^{\bm{l}}),F^{\bm{l}}_{w}$
		in place of $\mathbb{R}^{V^{\bm{l}}_{*}},F^{\bm{l}}_{w}|_{V^{\bm{l}^{n}}_{*}}$;
	\item Lemma \ref{lem:tsisg-RM-contraction} with $K^{\bm{l}^{|w|}}$ in place of $V^{\bm{l}^{|w|}}_{*}$;
	\item Proposition \ref{prop:tsisg-harmonic-Vn} with $\mathcal{C}(K^{\bm{l}}),F^{\bm{l}}_{w},K^{\bm{l}}_{w}$
		in place of $\mathbb{R}^{V^{\bm{l}}_{*}},F^{\bm{l}}_{w}|_{V^{\bm{l}^{n}}_{*}},F^{\bm{l}}_{w}(V^{\bm{l}^{n}}_{*})$;
	\item Proposition \ref{prop:tsisg-RM-completion}-\ref{it:tsisg-unifcont-Euc-RM},\ref{it:tsisg-unifcont-RM-Euc}
		for any $x,y\in K^{\bm{l}}$ and any $n\in\mathbb{N}$;
	\item\eqref{eq:tsisg-RM-diam-cell} for any $w\in W^{\bm{l}}_{*}$ and any $x\in K^{\bm{l}}_{w}$.
	\end{itemize}
\end{enumerate}
\end{rmk}

Finally, we can now consider $(\mathcal{E}^{\bm{l}},\mathcal{F}_{\bm{l}})$
as an irreducible, strongly local, regular symmetric Dirichlet form
over $K^{\bm{l}}$ as follows. See \cite[Sections 1.1 and 1.6]{FOT} or
\cite[Sections 1.1, 1.3 and 2.1]{CF} for the definitions of the relevant notions.
\begin{thm}\label{thm:tsisg-RF-DF}
Let $\mu$ be a \emph{Radon measure} on $K^{\bm{l}}$ with \emph{full support},
i.e., a Borel measure on $K^{\bm{l}}$ with $\mu(K^{\bm{l}})<\infty$
and $\mu(K^{\bm{l}}_{w})>0$ for any $w\in W^{\bm{l}}_{*}$.
Then $(\mathcal{E}^{\bm{l}},\mathcal{F}_{\bm{l}})$ is an irreducible,
strongly local regular symmetric Dirichlet form on $L^{2}(K^{\bm{l}},\mu)$.
\end{thm}
\begin{proof}
Since $\mathcal{C}(K^{\bm{l}})$ is dense in $L^{2}(K^{\bm{l}},\mu)$ by \cite[Theorem 3.14]{Rud},
$\mathcal{F}_{\bm{l}}$ is also dense in $L^{2}(K^{\bm{l}},\mu)$ by Theorem \ref{thm:tsisg-RF}-\ref{it:tsisg-RF-regular},
and then $(\mathcal{E}^{\bm{l}},\mathcal{F}_{\bm{l}})$ is a regular symmetric Dirichlet
form on $L^{2}(K^{\bm{l}},\mu)$ by Proposition \ref{prop:tsisg-RF}-\ref{it:RF2},
\eqref{eq:tsisg-RF-Hoelder}, Proposition \ref{prop:tsisg-RF}-\ref{it:RF5}
and Theorem \ref{thm:tsisg-RF}-\ref{it:tsisg-RF-regular},
strongly local by Theorem \ref{thm:tsisg-RF}-\ref{it:tsisg-RF-strongly-local}, and
irreducible by Proposition \ref{prop:tsisg-RF}-\ref{it:RF1} and \cite[Theorem 2.1.11]{CF}.
\qed\end{proof}
%
\section{Space-time scale function $\Psi_{\bm{l}}$ and \protect\hyperlink{fHKE}{$\textup{fHKE}(\Psi_{\bm{l}})$}}\label{sec:tsisg-fHKE}
In this section, we continue to fix an arbitrary
$\bm{l}=(l_{n})_{n=1}^{\infty}\in(\mathbb{N}\setminus\{1,2,3,4\})^{\mathbb{N}}$,
define a space-time scale function $\Psi_{\bm{l}}$ explicitly in terms of
$\bm{l}=(l_{n})_{n=1}^{\infty}$, and show that 
$(K^{\bm{l}},\refmet_{\bm{l}},\refmeas_{\bm{l}},\mathcal{E}^{\bm{l}},\mathcal{F}_{\bm{l}})$
satisfies \hyperlink{fHKE}{$\textup{fHKE}(\Psi_{\bm{l}})$}. First, $\Psi_{\bm{l}}$
is defined in a way analogous to \cite[(5.11)]{KM} for the usual scale irregular
Sierpi\'{n}ski gaskets but modified so as to take the ``asymptotically
one-dimensional'' nature of $K^{\bm{l}}$ into account, as follows.%
\begin{dfn}\label{dfn:tsisg-Psi}
We define a homeomorphism $\Psi_{\bm{l}}\colon[0,\infty)\to[0,\infty)$ by
\begin{align}
\Psi_{\bm{l}}(s)
	&:=\biggl(\frac{1}{M^{\bm{l}}_{n}}
			+\frac{s-1/L^{\bm{l}}_{n}}{1/L^{\bm{l}}_{n-1}-1/L^{\bm{l}}_{n}}
				\biggl(\frac{1}{M^{\bm{l}}_{n-1}}-\frac{1}{M^{\bm{l}}_{n}}\biggr)\biggr) \notag \\
	&\mspace{127.4mu}\cdot\biggl(R^{\bm{l}}_{n}
			+\frac{s-1/L^{\bm{l}}_{n}}{1/L^{\bm{l}}_{n-1}-1/L^{\bm{l}}_{n}}
				(R^{\bm{l}}_{n-1}-R^{\bm{l}}_{n})\biggr) \notag \\
&=\frac{1}{T^{\bm{l}}_{n}}\biggl(1+\frac{3l_{n}-4}{l_{n}-1}(L^{\bm{l}}_{n}s-1)\biggr)
	\biggl(1+\frac{\frac{2}{3}l_{n}-\frac{8}{9}}{l_{n}-1}(L^{\bm{l}}_{n}s-1)\biggr)
\label{eq:tsisg-Psi}
\end{align}
for $n\in\mathbb{N}$ and $s\in[1/L^{\bm{l}}_{n},1/L^{\bm{l}}_{n-1}]$ and
$\Psi_{\bm{l}}(s):=s^{\beta_{\bm{l},0}}$ for $s\in\{0\}\cup[1,\infty)$, where
$T^{\bm{l}}_{n}:=M^{\bm{l}}_{n}/R^{\bm{l}}_{n}
	=(\# S_{l_{1}}/r_{l_{1}})\cdots(\# S_{l_{n}}/r_{l_{n}})
	=\prod_{k=1}^{n}(2l_{k}^{2}-\frac{5}{3}l_{k}-\frac{1}{3})$
($T^{\bm{l}}_{0}:=1$) and $\beta_{\bm{l},0}:=\inf_{n\in\mathbb{N}}\beta_{l_{n}}$ with
$\beta_{l}:=\log_{l}(\# S_{l}/r_{l})=\log_{l}(2l^{2}-\frac{5}{3}l-\frac{1}{3})\in(2,2+\log_{5}2)$
for $l\in\mathbb{N}\setminus\{1,2,3,4\}$; note that $\{\beta_{l}\}_{l=5}^{\infty}$
is strictly decreasing and converges to $2$. We also set
$\beta_{\bm{l},1}:=\max_{n\in\mathbb{N}}\beta_{l_{n}}$, so that
$2\leq\beta_{\bm{l},0}\leq\beta_{\bm{l},1}\leq\beta_{5}<2+\log_{5}2$.
\end{dfn}
\begin{lem}\label{lem:tsisg-Psi}
$\Psi_{\bm{l}}$ satisfies \eqref{eq:Psi-ass} with $c_{\Psi}=81$,
$\beta_{0}=\beta_{\bm{l},0}$ and $\beta_{1}=\beta_{\bm{l},1}$.
\end{lem}
\begin{proof}
Let $r,R\in(0,\infty)$ satisfy $r\leq R$. If $r\geq 1$, then
$\Psi_{\bm{l}}(R)/\Psi_{\bm{l}}(r)=(R/r)^{\beta_{\bm{l},0}}\leq(R/r)^{\beta_{\bm{l},1}}$.
Next, if $r,R\in[1/L^{\bm{l}}_{n},1/L^{\bm{l}}_{n-1}]$
for some $n\in\mathbb{N}$, then we easily see from \eqref{eq:tsisg-Psi},
$1\leq L^{\bm{l}}_{n}r\leq L^{\bm{l}}_{n}R\leq l_{n}$ and
$l_{n}^{\beta_{l_{n}}-2}=2-\frac{5}{3}l_{n}^{-1}-\frac{1}{3}l_{n}^{-2}<2$ that
\begin{equation}\label{eq:lem:tsisg-Psi-proof}
\frac{1}{9}\biggl(\frac{R}{r}\biggr)^{\beta_{\bm{l},0}}
	<\frac{2}{9}l_{n}^{2-\beta_{l_{n}}}\biggl(\frac{R}{r}\biggr)^{\beta_{l_{n}}}
	\leq\frac{2}{9}\biggl(\frac{R}{r}\biggr)^{2}
	\leq\frac{\Psi_{\bm{l}}(R)}{\Psi_{\bm{l}}(r)}
	\leq\frac{9}{2}\biggl(\frac{R}{r}\biggr)^{2}
	\leq\frac{9}{2}\biggl(\frac{R}{r}\biggr)^{\beta_{\bm{l},1}}.
\end{equation}
Lastly, if $r<1$ and no such $n\in\mathbb{N}$ exists, then we can choose
$j,k\in\mathbb{N}\cup\{0\}$ with $j\leq k$ so that $r\in[1/L^{\bm{l}}_{k+1},1/L^{\bm{l}}_{k})$
and $R\in[1/L^{\bm{l}}_{j},1/L^{\bm{l}}_{j-1})$, where $1/L^{\bm{l}}_{-1}:=\infty$, and by
\eqref{eq:tsisg-Psi} and the definitions of $\beta_{\bm{l},0}$ and $\beta_{\bm{l},1}$ we have
\begin{equation*}
\frac{\Psi_{\bm{l}}(1/L^{\bm{l}}_{j})}{\Psi_{\bm{l}}(1/L^{\bm{l}}_{k})}
	=\frac{T^{\bm{l}}_{k}}{T^{\bm{l}}_{j}}
	=\prod_{n=j+1}^{k}\frac{\# S_{l_{n}}}{r_{l_{n}}}
	=\prod_{n=j+1}^{k}l_{n}^{\beta_{l_{n}}}
	\in\biggl[\biggl(\frac{1/L^{\bm{l}}_{j}}{1/L^{\bm{l}}_{k}}\biggr)^{\beta_{\bm{l},0}},
		\biggl(\frac{1/L^{\bm{l}}_{j}}{1/L^{\bm{l}}_{k}}\biggr)^{\beta_{\bm{l},1}}\biggr],
\end{equation*}
which together with \eqref{eq:lem:tsisg-Psi-proof} and the equality
\begin{equation*}
\frac{\Psi_{\bm{l}}(R)}{\Psi_{\bm{l}}(r)}
	=\frac{\Psi_{\bm{l}}(1/L^{\bm{l}}_{k})}{\Psi_{\bm{l}}(r)}
		\frac{\Psi_{\bm{l}}(1/L^{\bm{l}}_{j})}{\Psi_{\bm{l}}(1/L^{\bm{l}}_{k})}
		\frac{\Psi_{\bm{l}}(R)}{\Psi_{\bm{l}}(1/L^{\bm{l}}_{j})}
\end{equation*}
immediately yields \eqref{eq:Psi-ass} for $\Psi_{\bm{l}}$ with
$c_{\Psi}=81$, $\beta_{0}=\beta_{\bm{l},0}$ and $\beta_{1}=\beta_{\bm{l},1}$.
\qed\end{proof}

The main result of this section is the following theorem.
\begin{thm}\label{thm:tsisg-fHKE}
$(K^{\bm{l}},\refmet_{\bm{l}},\refmeas_{\bm{l}},\mathcal{E}^{\bm{l}},\mathcal{F}_{\bm{l}})$
satisfies \hyperlink{fHKE}{$\textup{fHKE}(\Psi_{\bm{l}})$}.
\end{thm}

The rest of this section is devoted to the proof of Theorem \ref{thm:tsisg-fHKE}.
We will conclude it from \cite[Theorem 15.10]{Kig12} by proving that
$(K^{\bm{l}},\refmet_{\bm{l}},\refmeas_{\bm{l}},\mathcal{E}^{\bm{l}},\mathcal{F}_{\bm{l}})$
satisfies the conditions $\textup{(DM1)}_{\Psi_{\bm{l}},\refmet_{\bm{l}}}$ and
$\textup{(DM2)}_{\Psi_{\bm{l}},\refmet_{\bm{l}}}$ defined in \cite[Definition 15.9-(3),(4)]{Kig12},
which are the central assumptions in \cite[Theorem 15.10]{Kig12}. A similar argument is
given in \cite[Chapter 24]{Kig12} for a large class of scale irregular Sierpi\'{n}ski gaskets,
but the possible unboundedness of $\bm{l}=(l_{n})_{n=1}^{\infty}$ requires some additional care.

The core of the proof of Theorem \ref{thm:tsisg-fHKE} is to establish the following proposition,
which is an extension of (the proof of) Proposition \ref{prop:tsisg-RM-completion}
to the case where $n\in\mathbb{N}$, $k\in\{1,\ldots,l_{n}\}$
and either $\refmet_{\bm{l}}(x,y)<k/L^{\bm{l}}_{n}$ or
$R_{\mathcal{E}^{\bm{l}}}(x,y)<\frac{1}{12}kR^{\bm{l}}_{n}$.
\begin{dfn}\label{dfn:Lambda-nwk-Unwk}
Let $n\in\mathbb{N}$ and $w\in W^{\bm{l}}_{n}$. For each $k\in\{0,\ldots,l_{n}\}$, we define
\begin{equation}\label{eq:Lambda-nwk}
\Lambda_{n,w}^{(k)}:=\biggl\{v\in W^{\bm{l}}_{n}\biggm|
	\begin{minipage}{228pt}
	there exists $\{v^{(j)}\}_{j=0}^{k}\subset W^{\bm{l}}_{n}$ such that $v^{(0)}=w$, $v^{(k)}=v$
	and $K^{\bm{l}}_{v^{(j-1)}}\cap K^{\bm{l}}_{v^{(j)}}\not=\emptyset$ for any $j\in\{1,\ldots,k\}$
	\end{minipage}
	\biggr\}
\end{equation}
($\Lambda_{n,w}^{(0)}:=\{w\}$) and $U_{n,w}^{(k)}:=\bigcup_{v\in\Lambda_{n,w}^{(k)}}K^{\bm{l}}_{v}$,
so that $2k+1\leq\#\Lambda_{n,w}^{(k)}\leq(6k)\vee 1$.
\end{dfn}
\begin{prop}\label{prop:Lambda-nwk-Unwk}
Let $n\in\mathbb{N}$, $w\in W^{\bm{l}}_{n}$, $x\in K^{\bm{l}}_{w}$ and $k\in\{1,\ldots,l_{n}\}$.
\begin{enumerate}[label=\textup{(\arabic*)},align=left,leftmargin=*]
\item\label{it:Unwk-in}If $y\in U_{n,w}^{(k)}$, then
	$\refmet_{\bm{l}}(x,y)<(k+2)/L^{\bm{l}}_{n}$ and
	$R_{\mathcal{E}^{\bm{l}}}(x,y)<(\frac{2}{3}k+\frac{10}{3})R^{\bm{l}}_{n}$.
\item\label{it:Unwk-out}If $y\in K^{\bm{l}}\setminus U_{n,w}^{(k)}$, then
	$\refmet_{\bm{l}}(x,y)\geq k/L^{\bm{l}}_{n}$ and
	$R_{\mathcal{E}^{\bm{l}}}(x,y)\geq\frac{1}{12}kR^{\bm{l}}_{n}$.
\item\label{it:Unwk-geodesic-adapted}$B_{\refmet_{\bm{l}}}(x,k/L^{\bm{l}}_{n})\subset U_{n,w}^{(k)}\subset B_{\refmet_{\bm{l}}}(x,(k+2)/L^{\bm{l}}_{n})$.
\item\label{it:Unwk-RM-adapted}$B_{R_{\mathcal{E}^{\bm{l}}}}(x,\frac{1}{12}kR^{\bm{l}}_{n})\subset U_{n,w}^{(k)}\subset B_{R_{\mathcal{E}^{\bm{l}}}}(x,(\frac{2}{3}k+\frac{10}{3})R^{\bm{l}}_{n})$.
\end{enumerate}
\end{prop}
\begin{proof}
\ref{it:Unwk-in} is immediate from \eqref{eq:tsisg-boundary},
the triangle inequality for $\refmet_{\bm{l}}$ and $R_{\mathcal{E}^{\bm{l}}}$,
\eqref{eq:tsisg-metric-diam-cell}, \eqref{eq:tsisg-RM-contraction-cell-vertices}
and \eqref{eq:tsisg-RM-diam-cell}. To see \ref{it:Unwk-out}, assume that
$y\not\in U_{n,w}^{(k)}$. For $\refmet_{\bm{l}}(x,y)$, by Proposition \ref{prop:tsisg-metric}
we can take $\gamma\colon[0,1]\to K^{\bm{l}}$ such that $\gamma(0)=x$, $\gamma(1)=y$ and
$\refmet_{\bm{l}}(\gamma(s),\gamma(t))=|s-t|\refmet_{\bm{l}}(x,y)$ for any $s,t\in[0,1]$,
and it then follows from \eqref{eq:tsisg-boundary} and $y\not\in U_{n,w}^{(k)}$ that
$\#(\gamma^{-1}(V^{\bm{l}}_{n})\cap(0,1))\geq k+1$, which yields
$\refmet_{\bm{l}}(x,y)=\ell_{\mathbb{R}^{2}}(\gamma)\geq k/L^{\bm{l}}_{n}$.
For $R_{\mathcal{E}^{\bm{l}}}(x,y)$, recalling Proposition \ref{prop:tsisg-harmonic-Vn},
define $u\in\mathbb{R}^{V^{\bm{l}}_{n}}$ by
\begin{equation}\label{eq:Unwk-out-proof}
u(z):=
	\begin{cases}
	1 & \textrm{if $z\in K^{\bm{l}}_{w}=U_{n,w}^{(0)}$,} \\
	1-\frac{j}{k} & \textrm{if $j\in\{1,\ldots,k\}$ and $z\in U_{n,w}^{(j)}\setminus U_{n,w}^{(j-1)}$,} \\
	0 & \textrm{if $z\in K^{\bm{l}}\setminus U_{n,w}^{(k)}$}
	\end{cases}
\end{equation}
for each $z\in V^{\bm{l}}_{n}$ and set $h_{n,w}^{(k)}:=h^{\bm{l}}_{n}(u)$, so that
$\mathcal{E}^{\bm{l}}(h_{n,w}^{(k)},h_{n,w}^{(k)})=\mathcal{E}^{\bm{l},n}(u,u)$,
$h_{n,w}^{(k)}|_{K^{\bm{l}}_{w}}=1$, $h_{n,w}^{(k)}|_{K^{\bm{l}}_{v}}=0$
for any $v\in W^{\bm{l}}_{n}\setminus\Lambda_{n,w}^{(k)}$,
and $u(F^{\bm{l}}_{v}(V_{0}))\subset\{1-\frac{j-1}{k},1-\frac{j}{k}\}$ for any
$j\in\{1,\ldots,k\}$ and any $v\in\Lambda_{n,w}^{(j)}\setminus\Lambda_{n,w}^{(j-1)}$.
Then combining these properties with Proposition \ref{prop:tsisg-RF}-\ref{it:RF4},
\eqref{eq:tsisg-form-ln}, \eqref{eq:tsisg-form-0} and $\#\Lambda_{n,w}^{(k)}\leq 6k$, we obtain
\begin{equation*}
R_{\mathcal{E}^{\bm{l}}}(x,y)
	\geq\frac{|h_{n,w}^{(k)}(x)-h_{n,w}^{(k)}(y)|^{2}}{\mathcal{E}^{\bm{l}}(h_{n,w}^{(k)},h_{n,w}^{(k)})}
	=\frac{1}{\mathcal{E}^{\bm{l},n}(u,u)}
	\geq\frac{\frac{1}{2}k^{2}R^{\bm{l}}_{n}}{\#\Lambda_{n,w}^{(k)}-1}
	\geq\frac{kR^{\bm{l}}_{n}}{12},
\end{equation*}
which proves \ref{it:Unwk-out}. Lastly, we also get \ref{it:Unwk-geodesic-adapted} and
\ref{it:Unwk-RM-adapted} since the conjunction of \ref{it:Unwk-geodesic-adapted} and
\ref{it:Unwk-RM-adapted} is clearly equivalent to that of \ref{it:Unwk-in} and \ref{it:Unwk-out}.
\qed\end{proof}

As an easy consequence of Proposition \ref{prop:Lambda-nwk-Unwk}, we further obtain the
following proposition, which contains $\textup{(DM2)}_{\Psi_{\bm{l}},\refmet_{\bm{l}}}$
as defined in \cite[Definition 15.9-(4)]{Kig12}.
\begin{prop}\label{prop:tsisg-DM2}
Let $x,y\in K^{\bm{l}}$, $s\in(0,\infty)$, $n\in\mathbb{N}$ and $k\in\{2,\ldots,l_{n}\}$.
\begin{enumerate}[label=\textup{(\arabic*)},align=left,leftmargin=*]
\item\label{it:tsisg-DM2-Psi-measure}If $s\in[(k-1)/L^{\bm{l}}_{n},k/L^{\bm{l}}_{n}]$, then
	\begin{equation}\label{eq:tsisg-DM2-Psi-measure}
	\frac{1}{18}\frac{k^{2}}{T^{\bm{l}}_{n}}
		\leq\Psi_{\bm{l}}(s)\leq\frac{9}{2}\frac{k^{2}}{T^{\bm{l}}_{n}}
	\mspace{28mu}\textrm{and}\mspace{28mu}
	\frac{7}{36}\frac{k}{M^{\bm{l}}_{n}}
		\leq\refmeas_{\bm{l}}(B_{\refmet_{\bm{l}}}(x,s))
		\leq 6\frac{k}{M^{\bm{l}}_{n}},
	\end{equation}
	whereas if $s\in[1,3]$, then $1\leq\Psi_{\bm{l}}(s)\leq 14$ and
	$\frac{7}{12}\leq\refmeas_{\bm{l}}(B_{\refmet_{\bm{l}}}(x,s))\leq 1$.
\item\label{it:tsisg-DM2-RM}If $\refmet_{\bm{l}}(x,y)\in[(k-1)/L^{\bm{l}}_{n},k/L^{\bm{l}}_{n})$, then
	\begin{equation}\label{eq:tsisg-DM2-RM}
	\frac{1}{48}kR^{\bm{l}}_{n}\leq R_{\mathcal{E}^{\bm{l}}}(x,y)\leq\frac{7}{3}kR^{\bm{l}}_{n},
	\end{equation}
	whereas $\refmet_{\bm{l}}(x,y)<3$, $R_{\mathcal{E}^{\bm{l}}}(x,y)<4$, and if
	$\refmet_{\bm{l}}(x,y)\geq 1$ then $R_{\mathcal{E}^{\bm{l}}}(x,y)\geq\frac{1}{14}$.
\item\label{it:tsisg-DM2}If $x\not=y$, then
	\begin{equation}\label{eq:tsisg-DM2}
	6^{-4}\frac{\Psi_{\bm{l}}(\refmet_{\bm{l}}(x,y))}{\refmeas_{\bm{l}}\bigl(B_{\refmet_{\bm{l}}}(x,\refmet_{\bm{l}}(x,y))\bigr)}
		\leq R_{\mathcal{E}^{\bm{l}}}(x,y)
		\leq 2^{8}\frac{\Psi_{\bm{l}}(\refmet_{\bm{l}}(x,y))}{\refmeas_{\bm{l}}\bigl(B_{\refmet_{\bm{l}}}(x,\refmet_{\bm{l}}(x,y))\bigr)}.
	\end{equation}
\end{enumerate}
\end{prop}
\begin{proof}
\begin{enumerate}[label=\textup{(\arabic*)},align=left,leftmargin=*]
\item Assume that $s\in[(k-1)/L^{\bm{l}}_{n},k/L^{\bm{l}}_{n}]$.
	By \eqref{eq:tsisg-Psi} and \eqref{eq:lem:tsisg-Psi-proof}  we have
	\begin{equation}\label{eq:tsisg-DM2-Psi}
	\begin{split}
	\frac{1}{18}\frac{k^{2}}{T^{\bm{l}}_{n}}
		\leq\frac{2}{9}\frac{(k-1)^{2}}{T^{\bm{l}}_{n}}
		=\frac{2}{9}(k-1)^{2}\Psi_{\bm{l}}(1/L^{\bm{l}}_{n})
		&\leq\Psi_{\bm{l}}((k-1)/L^{\bm{l}}_{n}) \\
	\leq\Psi_{\bm{l}}(s)
		\leq\Psi_{\bm{l}}(k/L^{\bm{l}}_{n})
		&\leq\frac{9}{2}k^{2}\Psi_{\bm{l}}(1/L^{\bm{l}}_{n})
		=\frac{9}{2}\frac{k^{2}}{T^{\bm{l}}_{n}}.
	\end{split}
	\end{equation}
	For $\refmeas_{\bm{l}}(B_{\refmet_{\bm{l}}}(x,s))$,
	choosing $w\in W^{\bm{l}}_{n}$ so that $x\in K^{\bm{l}}_{w}$, we see from
	Proposition \ref{prop:Lambda-nwk-Unwk}-\ref{it:Unwk-geodesic-adapted},
	\eqref{eq:tsisg-measure} and $2j+1\leq\#\Lambda_{n,w}^{(j)}\leq 6j$
	for $j\in\{1,\ldots,l_{n}\}$ that
	\begin{equation}\label{eq:tsisg-DM2-measure1}
	\refmeas_{\bm{l}}(B_{\refmet_{\bm{l}}}(x,s))
		\leq \refmeas_{\bm{l}}(B_{\refmet_{\bm{l}}}(x,k/L^{\bm{l}}_{n}))
		\leq \refmeas_{\bm{l}}(U_{n,w}^{(k)})
		=\frac{\#\Lambda_{n,w}^{(k)}}{M^{\bm{l}}_{n}}
		\leq 6\frac{k}{M^{\bm{l}}_{n}}
	\end{equation}
	and that, provided $k\geq 4$,
	\begin{equation}\label{eq:tsisg-DM2-measure2}
	\refmeas_{\bm{l}}(B_{\refmet_{\bm{l}}}(x,s))
		\geq \refmeas_{\bm{l}}(B_{\refmet_{\bm{l}}}(x,(k-1)/L^{\bm{l}}_{n}))
		\geq \refmeas_{\bm{l}}(U_{n,w}^{(k-3)})
		=\frac{\#\Lambda_{n,w}^{(k-3)}}{M^{\bm{l}}_{n}}
		\geq\frac{2k-5}{M^{\bm{l}}_{n}}.
	\end{equation}
	If $k\in\{2,3\}$, then choosing $v\in W^{\bm{l}}_{n+1}$ so that $x\in K^{\bm{l}}_{v}$,
	by Proposition \ref{prop:Lambda-nwk-Unwk}-\ref{it:Unwk-geodesic-adapted},
	\eqref{eq:tsisg-measure} and $\#\Lambda_{n+1,v}^{(l_{n+1}-2)}\geq 2l_{n+1}-3$ we get
	\begin{equation}\label{eq:tsisg-DM2-measure3}
	\begin{split}
	\refmeas_{\bm{l}}(B_{\refmet_{\bm{l}}}(x,s))
		&\geq \refmeas_{\bm{l}}(B_{\refmet_{\bm{l}}}(x,1/L^{\bm{l}}_{n}))
		= \refmeas_{\bm{l}}(B_{\refmet_{\bm{l}}}(x,l_{n+1}/L^{\bm{l}}_{n+1})) \\
	&\geq \refmeas_{\bm{l}}(U_{n+1,v}^{(l_{n+1}-2)})
		=\frac{\#\Lambda_{n+1,v}^{(l_{n+1}-2)}}{M^{\bm{l}}_{n+1}}
		\geq\frac{2l_{n+1}-3}{(3l_{n+1}-3)M^{\bm{l}}_{n}}
		\geq\frac{7}{12}\frac{1}{M^{\bm{l}}_{n}}.
	\end{split}
	\end{equation}
	\eqref{eq:tsisg-DM2-Psi}, \eqref{eq:tsisg-DM2-measure1}, \eqref{eq:tsisg-DM2-measure2}
	and \eqref{eq:tsisg-DM2-measure3} together yield \eqref{eq:tsisg-DM2-Psi-measure}.
	
	On the other hand, if $s\in[1,3]$, then
	$\Psi_{\bm{l}}(s)=s^{\beta_{\bm{l},0}}\in[1,3^{\beta_{5}}]\subset[1,14]$ and
	$1=\refmeas_{\bm{l}}(K^{\bm{l}})
		\geq\refmeas_{\bm{l}}(B_{\refmet_{\bm{l}}}(x,s))
		\geq\refmeas_{\bm{l}}(B_{\refmet_{\bm{l}}}(x,1))
		\geq\frac{7}{12}$
	by \eqref{eq:tsisg-DM2-measure3} with $n=0$.
\item Assume that $\refmet_{\bm{l}}(x,y)\in[(k-1)/L^{\bm{l}}_{n},k/L^{\bm{l}}_{n})$. Then
	by Proposition \ref{prop:Lambda-nwk-Unwk}-\ref{it:Unwk-geodesic-adapted},\ref{it:Unwk-RM-adapted},
	it follows from $\refmet_{\bm{l}}(x,y)<k/L^{\bm{l}}_{n}$ that
	$R_{\mathcal{E}^{\bm{l}}}(x,y)\leq(\frac{2}{3}k+\frac{10}{3})R^{\bm{l}}_{n}\leq\frac{7}{3}kR^{\bm{l}}_{n}$,
	from $\refmet_{\bm{l}}(x,y)\geq(k-1)/L^{\bm{l}}_{n}$ that
	$R_{\mathcal{E}^{\bm{l}}}(x,y)\geq\frac{1}{12}(k-3)R^{\bm{l}}_{n}\geq\frac{1}{48}kR^{\bm{l}}_{n}$
	provided $k\geq 4$, and from $\refmet_{\bm{l}}(x,y)\geq 1/L^{\bm{l}}_{n}=l_{n+1}/L^{\bm{l}}_{n+1}$
	that, provided $k\in\{2,3\}$,
	\begin{equation}\label{eq:tsisg-DM2-RM-proof}
	R_{\mathcal{E}^{\bm{l}}}(x,y)
		\geq\frac{1}{12}(l_{n+1}-2)R^{\bm{l}}_{n+1}
		=\frac{1}{12}\frac{l_{n+1}-2}{\frac{2}{3}l_{n+1}+\frac{1}{9}}R^{\bm{l}}_{n}
		\geq\frac{9}{124}R^{\bm{l}}_{n}
		>\frac{kR^{\bm{l}}_{n}}{48},
	\end{equation}
	proving \eqref{eq:tsisg-DM2-RM}.
	
	On the other hand, $\refmet_{\bm{l}}(x,y)<3$ by \eqref{eq:tsisg-metric-diam-cell},
	$R_{\mathcal{E}^{\bm{l}}}(x,y)<4$ by \eqref{eq:tsisg-RM-diam-cell},
	and if $\refmet_{\bm{l}}(x,y)\geq 1$ then
	$R_{\mathcal{E}^{\bm{l}}}(x,y)\geq\frac{9}{124}>\frac{1}{14}$
	by \eqref{eq:tsisg-DM2-RM-proof} with $n=0$.
\item This is immediate from \ref{it:tsisg-DM2-Psi-measure} and \ref{it:tsisg-DM2-RM}.
\qed\end{enumerate}
\end{proof}

We need the following definition and lemma for the proof of the other condition
$\textup{(DM1)}_{\Psi_{\bm{l}},\refmet_{\bm{l}}}$ required to apply \cite[Theorem 15.10]{Kig12}.
\begin{dfn}\label{dfn:tsisg-PsiMR}
We define homeomorphisms $\Psi^{\mathrm{M}}_{\bm{l}},\Psi^{\mathrm{R}}_{\bm{l}}\colon[0,\infty)\to[0,\infty)$ by
\begin{align}
\Psi^{\mathrm{M}}_{\bm{l}}(s)
	&:=\frac{1}{M^{\bm{l}}_{n}}
		+\frac{(s-1/L^{\bm{l}}_{n})(1/M^{\bm{l}}_{n-1}-1/M^{\bm{l}}_{n})}{1/L^{\bm{l}}_{n-1}-1/L^{\bm{l}}_{n}}
	=\frac{1}{M^{\bm{l}}_{n}}\biggl(1+\frac{3l_{n}-4}{l_{n}-1}(L^{\bm{l}}_{n}s-1)\biggr), \notag \\
\Psi^{\mathrm{R}}_{\bm{l}}(s)
	&:=R^{\bm{l}}_{n}
		+\frac{(s-1/L^{\bm{l}}_{n})(R^{\bm{l}}_{n-1}-R^{\bm{l}}_{n})}{1/L^{\bm{l}}_{n-1}-1/L^{\bm{l}}_{n}}
	=R^{\bm{l}}_{n}\biggl(1+\frac{\frac{2}{3}l_{n}-\frac{8}{9}}{l_{n}-1}(L^{\bm{l}}_{n}s-1)\biggr)
\label{eq:tsisg-PsiMR}
\end{align}
for $n\in\mathbb{N}$ and $s\in[1/L^{\bm{l}}_{n},1/L^{\bm{l}}_{n-1}]$ and
$\Psi^{\mathrm{M}}_{\bm{l}}(s):=s^{\beta^{\mathrm{M}}_{\bm{l},0}}$ and
$\Psi^{\mathrm{R}}_{\bm{l}}(s):=s^{\beta^{\mathrm{R}}_{\bm{l},1}}$
for $s\in\{0\}\cup[1,\infty)$, where
$\beta^{\mathrm{M}}_{\bm{l},0}:=\inf_{n\in\mathbb{N}}\log_{l_{n}}\# S_{l_{n}}$ and
$\beta^{\mathrm{R}}_{\bm{l},1}:=-\inf_{n\in\mathbb{N}}\log_{l_{n}}r_{l_{n}}$
(note that $\{\log_{l}\# S_{l}\}_{l=5}^{\infty}$ and $\{\log_{l}r_{l}\}_{l=5}^{\infty}$
are strictly decreasing), so that
$\Psi_{\bm{l}}=\Psi^{\mathrm{M}}_{\bm{l}}\Psi^{\mathrm{R}}_{\bm{l}}$. We also set
$\beta^{\mathrm{M}}_{\bm{l},1}:=\max_{n\in\mathbb{N}}\log_{l_{n}}\# S_{l_{n}}$ and
$\beta^{\mathrm{R}}_{\bm{l},0}:=-\max_{n\in\mathbb{N}}\log_{l_{n}}r_{l_{n}}$.
\end{dfn}
\begin{lem}\label{lem:tsisg-PsiMR}
\begin{enumerate}[label=\textup{(\arabic*)},align=left,leftmargin=*]
\item\label{it:tsisg-PsiM}$\Psi^{\mathrm{M}}_{\bm{l}}$ satisfies \eqref{eq:Psi-ass} with $c_{\Psi}=81$,
	$\beta_{0}=\beta^{\mathrm{M}}_{\bm{l},0}$ and $\beta_{1}=\beta^{\mathrm{M}}_{\bm{l},1}$.
\item\label{it:tsisg-PsiR}$\Psi^{\mathrm{R}}_{\bm{l}}$ satisfies \eqref{eq:Psi-ass} with $c_{\Psi}=6$,
	$\beta_{0}=\beta^{\mathrm{R}}_{\bm{l},0}$ and $\beta_{1}=\beta^{\mathrm{R}}_{\bm{l},1}$.
\end{enumerate}
\end{lem}
\begin{proof}
These are proved in exactly the same way as Lemma \ref{lem:tsisg-Psi}.
\qed\end{proof}

Finally, $\textup{(DM1)}_{\Psi_{\bm{l}},\refmet_{\bm{l}}}$ as defined in
\cite[Definition 15.9-(3)]{Kig12} is deduced as follows.%
\begin{prop}\label{prop:tsisg-DM1}
Let $x,y\in K^{\bm{l}}$ and $s\in(0,3]$. Then
\begin{gather}\label{eq:tsisg-measure-PsiM}
\frac{1}{16}\Psi^{\mathrm{M}}_{\bm{l}}(s)
	\leq \refmeas_{\bm{l}}(B_{\refmet_{\bm{l}}}(x,s))
	\leq 12\Psi^{\mathrm{M}}_{\bm{l}}(s),
	\mspace{10mu}
	\frac{1}{12}\Psi^{\mathrm{R}}_{\bm{l}}(s)
	\leq\frac{\Psi_{\bm{l}}(s)}{\refmeas_{\bm{l}}(B_{\refmet_{\bm{l}}}(x,s))}
	\leq 16\Psi^{\mathrm{R}}_{\bm{l}}(s), \\
2^{-14}\Psi^{\mathrm{R}}_{\bm{l}}(\refmet_{\bm{l}}(x,y))
	\leq R_{\mathcal{E}^{\bm{l}}}(x,y)
	\leq 2^{12}\Psi^{\mathrm{R}}_{\bm{l}}(\refmet_{\bm{l}}(x,y)).
\label{eq:tsisg-RM-qs}
\end{gather}
In particular, if $x\not=y$, then for any $\lambda\in(0,1]$,
\begin{equation}\label{eq:tsisg-DM1}
\frac{6^{-4}\lambda^{\beta^{\mathrm{R}}_{\bm{l},1}}\Psi_{\bm{l}}(\refmet_{\bm{l}}(x,y))}{\refmeas_{\bm{l}}\bigl(B_{\refmet_{\bm{l}}}(x,\refmet_{\bm{l}}(x,y))\bigr)}
	\leq\frac{\Psi_{\bm{l}}(\lambda \refmet_{\bm{l}}(x,y))}{\refmeas_{\bm{l}}\bigl(B_{\refmet_{\bm{l}}}(x,\lambda \refmet_{\bm{l}}(x,y))\bigr)}
	\leq\frac{6^{4}\lambda^{\beta^{\mathrm{R}}_{\bm{l},0}}\Psi_{\bm{l}}(\refmet_{\bm{l}}(x,y))}{\refmeas_{\bm{l}}\bigl(B_{\refmet_{\bm{l}}}(x,\refmet_{\bm{l}}(x,y))\bigr)}.
\end{equation}
\end{prop}
\begin{proof}
If $n\in\mathbb{N}$, $k\in\{2,\ldots,l_{n}\}$ and $s\in[(k-1)/L^{\bm{l}}_{n},kL^{\bm{l}}_{n}]$ then we easily
see from \eqref{eq:tsisg-PsiMR} that $\frac{1}{2}k/M^{\bm{l}}_{n}\leq\Psi^{\mathrm{M}}_{\bm{l}}(s)\leq 3k/M^{\bm{l}}_{n}$,
and if $s\in[1,3]$ then $\Psi^{\mathrm{M}}_{\bm{l}}(s)=s^{\beta^{\mathrm{M}}_{\bm{l},0}}\in[1,3^{\log_{5}\# S_{5}}]\subset[1,6]$.
These facts, Proposition \ref{prop:tsisg-DM2}-\ref{it:tsisg-DM2-Psi-measure} and
$\Psi_{\bm{l}}=\Psi^{\mathrm{M}}_{\bm{l}}\Psi^{\mathrm{R}}_{\bm{l}}$ together imply
\eqref{eq:tsisg-measure-PsiM}, which in turn in combination with
Proposition \ref{prop:tsisg-DM2}-\ref{it:tsisg-DM2} and
Lemma \ref{lem:tsisg-PsiMR}-\ref{it:tsisg-PsiR}, respectively, yields
\eqref{eq:tsisg-RM-qs} and \eqref{eq:tsisg-DM1} since
$\refmet_{\bm{l}}(x,y)\in[0,3)$ by \eqref{eq:tsisg-metric-diam-cell}.
\qed\end{proof}
\begin{proof}[\hspace*{-1.45pt}of Theorem \textup{\ref{thm:tsisg-fHKE}}]
By Propositions \ref{prop:tsisg-metric}, \ref{prop:tsisg-RM-completion} and Theorem \ref{thm:tsisg-RF},
$(\mathcal{E}^{\bm{l}},\mathcal{F}_{\bm{l}})$ is a strongly local regular resistance form on
$K^{\bm{l}}$ whose resistance metric $R_{\mathcal{E}^{\bm{l}}}$ gives the same topology
as the geodesic metric $\refmet_{\bm{l}}$. We also have \textup{(ACC)} as defined in
\cite[Definition 7.4]{Kig12} by \cite[Proposition 7.6]{Kig12},
$\textup{(DM1)}_{\Psi_{\bm{l}},\refmet_{\bm{l}}}$ by \eqref{eq:tsisg-DM1} and
$\beta^{\mathrm{R}}_{\bm{l},0}>0$, and $\textup{(DM2)}_{\Psi_{\bm{l}},\refmet_{\bm{l}}}$
by Proposition \ref{prop:tsisg-DM2}-\ref{it:tsisg-DM2}.
Thus \cite[Theorem 15.10, Cases 1 and~2]{Kig12} are applicable to
$(K^{\bm{l}},\refmet_{\bm{l}},\refmeas_{\bm{l}},\mathcal{E}^{\bm{l}},\mathcal{F}_{\bm{l}})$
and imply that it satisfies \hyperlink{fHKE}{$\textup{fHKE}(\Psi_{\bm{l}})$}.
\qed\end{proof}
%
\section{Singularity of the energy measures}\label{sec:tsisg-sing}
As in the previous two sections, we fix an arbitrary
$\bm{l}=(l_{n})_{n=1}^{\infty}\in(\mathbb{N}\setminus\{1,2,3,4\})^{\mathbb{N}}$ throughout
this section. We first recall the definition of the $\mathcal{E}^{\bm{l}}$-energy measures.%
\begin{dfn}[$\mathcal{E}^{\bm{l}}$-energy measure; {\cite[(3.2.14)]{FOT}}]\label{dfn:EnergyMeas}
Let $u\in\mathcal{F}_{\bm{l}}$. We define the \emph{$\mathcal{E}^{\bm{l}}$-energy measure}
$\mu^{\bm{l}}_{\langle u\rangle}$ of $u$ as the unique Borel measure on $K^{\bm{l}}$ such that
\begin{equation}\label{eq:EnergyMeas}
\int_{K^{\bm{l}}}f\,d\mu^{\bm{l}}_{\langle u\rangle}
	=\mathcal{E}^{\bm{l}}(u,fu)-\frac{1}{2}\mathcal{E}^{\bm{l}}(u^{2},f)
	\qquad\textrm{for any $f\in\mathcal{F}_{\bm{l}}$;}
\end{equation}
since $\mathcal{F}_{\bm{l}}$ is a dense subalgebra of $(\mathcal{C}(K^{\bm{l}}),\|\cdot\|_{\sup})$
by Theorem \ref{thm:tsisg-RF}-\ref{it:tsisg-RF-regular} and
\begin{equation}\label{eq:EnergyMeas-exists}
0\leq\mathcal{E}^{\bm{l}}(u,f^{+}u)-\frac{1}{2}\mathcal{E}^{\bm{l}}(u^{2},f^{+})
	\leq\|f\|_{\sup}\mathcal{E}^{\bm{l}}(u,u)
	\qquad\textrm{for any $f\in\mathcal{F}_{\bm{l}}$}
\end{equation}
by \eqref{eq:tsisg-form-0}, \eqref{eq:tsisg-form-ln} and \eqref{eq:tsisg-RF-form},
such $\mu^{\bm{l}}_{\langle u\rangle}$ exists and is unique by the
Riesz(--Markov--Kakutani) representation theorem (see, e.g., \cite[Theorems 2.14 and 2.18]{Rud}).
\end{dfn}
Proposition \ref{prop:tsisg-RF-self-sim} yields the following alternative
characterization of $\mu^{\bm{l}}_{\langle u\rangle}$.
\begin{prop}\label{prop:EnergyMeas-cell}
Let $u\in\mathcal{F}_{\bm{l}}$. Then $\mu^{\bm{l}}_{\langle u\rangle}(\{x\})=0$
for any $x\in K^{\bm{l}}$. Moreover, $\mu^{\bm{l}}_{\langle u\rangle}$
is the unique Borel measure on $K^{\bm{l}}$ such that
\begin{equation}\label{eq:EnergyMeas-cell}
\mu^{\bm{l}}_{\langle u\rangle}(K^{\bm{l}}_{w})
	=\frac{1}{R^{\bm{l}}_{|w|}}\mathcal{E}^{\bm{l}^{|w|}}(u\circ F^{\bm{l}}_{w},u\circ F^{\bm{l}}_{w})
	\qquad\textrm{for any $w\in W^{\bm{l}}_{*}$.}
\end{equation}
\end{prop}
\begin{proof}
Since $(\mathcal{E}^{\bm{l}},\mathcal{F}_{\bm{l}})$ is a strongly local regular symmetric
Dirichlet form on $L^{2}(K^{\bm{l}},\refmeas_{\bm{l}})$ by Theorem \ref{thm:tsisg-RF-DF},
the Borel measure $\mu^{\bm{l}}_{\langle u\rangle}(u^{-1}(\cdot))$ on $\mathbb{R}$ is absolutely
continuous with respect to the Lebesgue measure on $\mathbb{R}$ by \cite[Theorem 4.3.8]{CF}, and
therefore $\mu^{\bm{l}}_{\langle u\rangle}(\{x\})\leq\mu^{\bm{l}}_{\langle u\rangle}(u^{-1}(u(x)))=0$
for any $x\in K^{\bm{l}}$.

The uniqueness of a Borel measure on $K^{\bm{l}}$ satisfying \eqref{eq:EnergyMeas-cell}
is immediate from \eqref{eq:tsisg-RF-self-sim-form} and the Dynkin class theorem
(see, e.g., \cite[Appendixes, Theorem 4.2]{EK}). To show that
$\mu^{\bm{l}}_{\langle u\rangle}$ has the property \eqref{eq:EnergyMeas-cell},
let $n,k\in\mathbb{N}\cup\{0\}$, $w\in W^{\bm{l}}_{n}$ and set
$f_{k}:=h^{\bm{l}}_{n+k}(\one_{K^{\bm{l}}_{w}\cap V^{\bm{l}}_{n+k}})$, so that
$\one_{K^{\bm{l}}_{w}}\leq f_{k}\leq
	\one_{K^{\bm{l}}_{w}\cup\bigcup_{q\in F^{\bm{l}}_{w}(V_{0})}B_{\refmet_{\bm{l}}}(q,2/L^{\bm{l}}_{n+k})}$
by Proposition \ref{prop:tsisg-harmonic-Vn} and \eqref{eq:tsisg-metric-diam-cell}.
Then from \eqref{eq:EnergyMeas-exists}, \eqref{eq:tsisg-RF-self-sim-form} and
\eqref{eq:EnergyMeas} we obtain
\begin{align*}
\frac{\mathcal{E}^{\bm{l}^{n}}(u\circ F^{\bm{l}}_{w},u\circ F^{\bm{l}}_{w})}{R^{\bm{l}}_{n}}
	&=\frac{1}{R^{\bm{l}}_{n}}\Bigl(\mathcal{E}^{\bm{l}^{n}}(u\circ F^{\bm{l}}_{w},(f_{k}u)\circ F^{\bm{l}}_{w})
		-\frac{1}{2}\mathcal{E}^{\bm{l}^{n}}((u^{2})\circ F^{\bm{l}}_{w},f_{k}\circ F^{\bm{l}}_{w})\Bigr) \\
&\leq\mathcal{E}^{\bm{l}}(u,f_{k}u)-\frac{1}{2}\mathcal{E}^{\bm{l}}(u^{2},f_{k})
	=\int_{K^{\bm{l}}}f_{k}\,d\mu^{\bm{l}}_{\langle u\rangle} \\
&\leq\mu^{\bm{l}}_{\langle u\rangle}\Bigl(K^{\bm{l}}_{w}\cup\bigcup\nolimits_{q\in F^{\bm{l}}_{w}(V_{0})}B_{\refmet_{\bm{l}}}(q,2/L^{\bm{l}}_{n+k})\Bigr)
	\xrightarrow{k\to\infty}\mu^{\bm{l}}_{\langle u\rangle}(K^{\bm{l}}_{w})
\end{align*}
and hence
$\mathcal{E}^{\bm{l}^{n}}(u\circ F^{\bm{l}}_{w},u\circ F^{\bm{l}}_{w})/R^{\bm{l}}_{n}
	\leq\mu^{\bm{l}}_{\langle u\rangle}(K^{\bm{l}}_{w})$,
where the equality necessarily holds since the sum over
$w\in W^{\bm{l}}_{n}$ of each side of this inequality is equal to
$\mathcal{E}^{\bm{l}}(u,u)=\mu^{\bm{l}}_{\langle u\rangle}(K^{\bm{l}})$
by \eqref{eq:tsisg-RF-self-sim-form}, \eqref{eq:tsisg-boundary},
$\mu^{\bm{l}}_{\langle u\rangle}(V^{\bm{l}}_{n})=0$ and
\eqref{eq:EnergyMeas} with $f=\one_{K^{\bm{l}}}$.
\qed\end{proof}

The purpose of this section is to prove the following theorem.
\begin{thm}\label{thm:tsisg-sing}
$\mu^{\bm{l}}_{\langle u\rangle}\perp \refmeas_{\bm{l}}$ for any $u\in\mathcal{F}_{\bm{l}}$.
\end{thm}

The rest of this section is devoted to the proof of Theorem \ref{thm:tsisg-sing}.
First, we observe that the proof is reduced to the case of
$u\in\bigcup_{n=0}^{\infty}\mathcal{H}_{\bm{l},n}$ by the following two lemmas.
\begin{lem}\label{lem:tsisg-harmonic-approx}
Let $u\in\mathcal{F}_{\bm{l}}$ and set $u_{n}:=h^{\bm{l}}_{n}(u|_{V^{\bm{l}}_{n}})$ for each
$n\in\mathbb{N}\cup\{0\}$ \textup{(recall Proposition \ref{prop:tsisg-harmonic-Vn})}. Then
$\mathcal{E}^{\bm{l}}(u-u_{n},u-u_{n})
	=\mathcal{E}^{\bm{l}}(u,u)-\mathcal{E}^{\bm{l},n}(u|_{V^{\bm{l}}_{n}},u|_{V^{\bm{l}}_{n}})$
for any $n\in\mathbb{N}\cup\{0\}$. In particular,
$\lim_{n\to\infty}\mathcal{E}^{\bm{l}}(u-u_{n},u-u_{n})=0$.
\end{lem}
\begin{proof}
We follow \cite[Proof of Lemma 3.2.17]{Kig01}. Let $n\in\mathbb{N}\cup\{0\}$. Then
$\mathcal{E}^{\bm{l}}(u_{n},u)
	=\mathcal{E}^{\bm{l}}(u_{n},u_{n})
	=\mathcal{E}^{\bm{l},n}(u|_{V^{\bm{l}}_{n}},u|_{V^{\bm{l}}_{n}})$
by $u_{n}\in\mathcal{H}_{\bm{l},n}$, \eqref{eq:tsisg-harmonic-Vn} and
Proposition \ref{prop:tsisg-harmonic-Vn} and thus
$\mathcal{E}^{\bm{l}}(u-u_{n},u-u_{n})
	=\mathcal{E}^{\bm{l}}(u,u)-\mathcal{E}^{\bm{l},n}(u|_{V^{\bm{l}}_{n}},u|_{V^{\bm{l}}_{n}})$,
which converges to $0$ as $n\to\infty$ by \eqref{eq:tsisg-RF-form}.
\qed\end{proof}
\begin{lem}\label{lem:tsisg-sing-lin-cont}
If a Borel measure $\mu$ on $K^{\bm{l}}$, $\{u_{n}\}_{n=1}^{\infty}\subset\mathcal{F}_{\bm{l}}$
and $u\in\mathcal{F}_{\bm{l}}$ satisfy $\lim_{n\to\infty}\mathcal{E}^{\bm{l}}(u-u_{n},u-u_{n})=0$
and $\mu^{\bm{l}}_{\langle u_{n}\rangle}\perp\mu$ for any $n\in\mathbb{N}$, then
$\mu^{\bm{l}}_{\langle u\rangle}\perp\mu$.
\end{lem}
\begin{proof}
This is a special case of \cite[Lemma 3.7-(b)]{KM}, whose proof works for any regular symmetric Dirichlet space.
\qed\end{proof}

To prove that $\mu^{\bm{l}}_{\langle h\rangle}\perp \refmeas_{\bm{l}}$
for any $h\in\bigcup_{n=0}^{\infty}\mathcal{H}_{\bm{l},n}$, noting that
$\mathcal{H}_{\bm{l},0}\subset\mathcal{H}_{\bm{l},1}$ and recalling Proposition \ref{prop:tsisg-harmonic-Vn},
in the following lemma we calculate explicitly the matrix representation of the linear maps
$\mathbb{R}^{V_{0}}\ni u\mapsto h^{\bm{l}}_{0}(u)\circ F^{\bm{l}}_{i}|_{V_{0}}\in\mathbb{R}^{V_{0}}$,
$i\in W^{\bm{l}}_{1}=S_{l_{1}}$, which we identify with the linear maps
$\mathcal{H}_{\bm{l},0}\ni h\mapsto h\circ F^{\bm{l}}_{i}\in\mathcal{H}_{\bm{l}^{1},0}$.%
\begin{lem}\label{lem:tsisg-harmonic-V1-values}
Set $l:=l_{1}$, $a_{l}:=\frac{1}{9}r_{l}=(6l+1)^{-1}$, and for each $i\in S_{l}$
let $A^{l}_{i}$ denote the matrix representation of the linear map
$\mathbb{R}^{V_{0}}\ni u\mapsto h^{\bm{l}}_{0}(u)\circ F^{\bm{l}}_{i}|_{V_{0}}\in\mathbb{R}^{V_{0}}$
with respect to the basis $(\one_{q_{0}},\one_{q_{1}},\one_{q_{2}})$ of $\mathbb{R}^{V_{0}}$.
Then for any $k\in\{2,\ldots,l-3\}$,
\begin{equation}\label{eq:tsisg-harmonic-V1-values}
\begin{split}
A^{l}_{(k,0)}&=
	\begin{pmatrix}
	1-(6k+3)a_{l} & (6k-2)a_{l} & 5a_{l} \\
	1-(6k+9)a_{l} & (6k+4)a_{l} & 5a_{l} \\
	1-(6k+6)a_{l} & (6k+1)a_{l} & 5a_{l}
	\end{pmatrix}, \\
A^{l}_{(0,k)}&=
	\begin{pmatrix}
	1-(6k+3)a_{l} & 5a_{l} & (6k-2)a_{l} \\
	1-(6k+6)a_{l} & 5a_{l} & (6k+1)a_{l} \\
	1-(6k+9)a_{l} & 5a_{l} & (6k+4)a_{l}
	\end{pmatrix}, \\
A^{l}_{(l-1-k,k)}&=
	\begin{pmatrix}
	5a_{l} & 1-(6k+6)a_{l} & (6k+1)a_{l} \\
	5a_{l} & 1-(6k+3)a_{l} & (6k-2)a_{l} \\
	5a_{l} & 1-(6k+9)a_{l} & (6k+4)a_{l}
	\end{pmatrix}.
\end{split}
\end{equation}
\end{lem}
\begin{proof}
This follows by solving the linear equation in
$(h(x))_{x\in V^{l}_{1}\setminus V_{0}}$ for $h=h^{\bm{l}}_{0}(u)$ from
Proposition \ref{prop:tsisg-harmonic-Vn}-\ref{it:tsisg-harmonic-Vn-linear-eq}
with $(n,k)=(0,1)$ under $h(q)=u(q)$ for $q\in V_{0}$.
\qed\end{proof}

Our proof that $\mu^{\bm{l}}_{\langle h\rangle}\perp \refmeas_{\bm{l}}$ for
$h\in\bigcup_{n=0}^{\infty}\mathcal{H}_{\bm{l},n}$ is based on the following fact.
\begin{thm}[{\cite[Theorem 4.1]{Hin05}}]\label{thm:prob-meas-filtration-singular}
Let $(\Omega,\mathscr{F},\mathbb{P})$ be a probability space and let
$\{\mathscr{F}_{n}\}_{n=0}^{\infty}$ be a non-decreasing sequence of $\sigma$-algebras
in $\Omega$ such that $\bigcup_{n=0}^{\infty}\mathscr{F}_{n}$ generates $\mathscr{F}$.
Let $\widetilde{\mathbb{P}}$ be a probability measure on $(\Omega,\mathscr{F})$ such that
$\widetilde{\mathbb{P}}|_{\mathscr{F}_{n}}\ll\mathbb{P}|_{\mathscr{F}_{n}}$
for any $n\in\mathbb{N}\cup\{0\}$, and for each $n\in\mathbb{N}$ define
$\alpha_{n}\in L^{1}(\Omega,\mathscr{F}_{n},\mathbb{P}|_{\mathscr{F}_{n}})$ by
\begin{equation}\label{eq:prob-meas-filtration-singular-alphan}
\alpha_{n}:=
	\begin{cases}
	\dfrac{d(\widetilde{\mathbb{P}}|_{\mathscr{F}_{n}})/d(\mathbb{P}|_{\mathscr{F}_{n}})}{d(\widetilde{\mathbb{P}}|_{\mathscr{F}_{n-1}})/d(\mathbb{P}|_{\mathscr{F}_{n-1}})}
		& \textrm{on $\{d(\widetilde{\mathbb{P}}|_{\mathscr{F}_{n-1}})/d(\mathbb{P}|_{\mathscr{F}_{n-1}})>0\}$,} \\
	0 & \textrm{on $\{d(\widetilde{\mathbb{P}}|_{\mathscr{F}_{n-1}})/d(\mathbb{P}|_{\mathscr{F}_{n-1}})=0\}$,}
	\end{cases}
\end{equation}
so that $\mathbb{E}[\sqrt{\alpha_{n}}\mid\mathscr{F}_{n-1}]\leq 1$
$\mathbb{P}|_{\mathscr{F}_{n-1}}$-a.s.\ by conditional Jensen's inequality,
where $\mathbb{E}[\cdot\mid\mathscr{F}_{n-1}]$ denotes the conditional
expectation given $\mathscr{F}_{n-1}$ with respect to $\mathbb{P}$. If
\begin{equation}\label{eq:prob-meas-filtration-singular}
\sum_{n=1}^{\infty}(1-\mathbb{E}[\sqrt{\alpha_{n}}\mid\mathscr{F}_{n-1}])=\infty
	\mspace{28mu}\textrm{$\mathbb{P}$-a.s.,}
\end{equation}
then $\widetilde{\mathbb{P}}\perp\mathbb{P}$.
\end{thm}
We will apply Theorem \ref{thm:prob-meas-filtration-singular} under the setting of
the following lemma with $\mathbb{P}=\refmeas_{\bm{l}}$.%
\begin{lem}\label{lem:prob-meas-filtration-singular-tsisg}
Set $\Omega:=K^{\bm{l}}$, $\mathscr{F}:=\mathscr{B}(K^{\bm{l}})$ and let
$\mathbb{P},\widetilde{\mathbb{P}}$ be probability measures on $(\Omega,\mathscr{F})$
such that $\mathbb{P}(K^{\bm{l}}_{w})>0$ for any $w\in W^{\bm{l}}_{*}$ and
$\mathbb{P}(V^{\bm{l}}_{*})=\widetilde{\mathbb{P}}(V^{\bm{l}}_{*})=0$. Set
$\mathscr{F}_{n}:=\{A\cup\bigcup_{w\in\Lambda}(K^{\bm{l}}_{w}\setminus V^{\bm{l}}_{n})
	\mid\textrm{$\Lambda\subset W^{\bm{l}}_{n}$, $A\subset V^{\bm{l}}_{n}$}\}$
for each $n\in\mathbb{N}\cup\{0\}$, so that $\{\mathscr{F}_{n}\}_{n=0}^{\infty}$
is a non-decreasing sequence of $\sigma$-algebras in $\Omega$ by \eqref{eq:tsisg-boundary},
$\bigcup_{n=0}^{\infty}\mathscr{F}_{n}$ generates $\mathscr{F}$, and
$\widetilde{\mathbb{P}}|_{\mathscr{F}_{n}}\ll\mathbb{P}|_{\mathscr{F}_{n}}$
for any $n\in\mathbb{N}\cup\{0\}$. Let $n\in\mathbb{N}$
and define $\alpha_{n}\in L^{1}(\Omega,\mathscr{F}_{n},\mathbb{P}|_{\mathscr{F}_{n}})$
by \eqref{eq:prob-meas-filtration-singular-alphan}. Then for each $w\in W^{\bm{l}}_{n-1}$,
\begin{equation}\label{eq:prob-meas-filtration-singular-tsisg}
\mathbb{E}[\sqrt{\alpha_{n}}\mid\mathscr{F}_{n-1}]|_{K^{\bm{l}}_{w}\setminus V^{\bm{l}}_{n-1}}=
	\begin{cases}
	\displaystyle\sum_{i\in S_{l_{n}}}\sqrt{\frac{\widetilde{\mathbb{P}}(K^{\bm{l}}_{wi})}{\widetilde{\mathbb{P}}(K^{\bm{l}}_{w})}}
		\sqrt{\frac{\mathbb{P}(K^{\bm{l}}_{wi})}{\mathbb{P}(K^{\bm{l}}_{w})}} & \textrm{if $\widetilde{\mathbb{P}}(K^{\bm{l}}_{w})>0$,} \\
	0 & \textrm{if $\widetilde{\mathbb{P}}(K^{\bm{l}}_{w})=0$.}
	\end{cases}
\end{equation}
\end{lem}
\begin{proof}
This follows easily by direct calculations based on
\eqref{eq:prob-meas-filtration-singular-alphan} and \eqref{eq:tsisg-boundary}.
\qed\end{proof}

The following proposition is the key step of the proof of Theorem \ref{thm:tsisg-sing}.
\begin{prop}\label{prop:muh-filtration-singular-tsisg}
Let $k\in\mathbb{N}\cup\{0\}$, $h\in\mathcal{H}_{\bm{l},k}$, $x\in K^{\bm{l}}\setminus V^{\bm{l}}_{*}$, and
let $\omega^{x}=(\omega^{x}_{n})_{n=1}^{\infty}$ be the element of $\prod_{n=1}^{\infty}S_{l_{n}}$, unique by
\eqref{eq:tsisg-boundary}, such that $\{x\}=\bigcap_{n=1}^{\infty}K^{\bm{l}}_{\omega^{x}_{1}\ldots\omega^{x}_{n}}$.
Let $n\in\mathbb{N}\cap[k+2,\infty)$ and assume that
$\mu^{\bm{l}}_{\langle h\rangle}\bigl(K^{\bm{l}}_{\omega^{x}_{1}\ldots\omega^{x}_{n-1}}\bigr)>0$
and that $\omega^{x}_{n-1}\in S_{l_{n-1},1}$, where
$S_{l,1}:=\{(i_{1},i_{2})\in S_{l}\mid i_{1}\vee i_{2}\in\{2,\ldots,l-3\}\}$
for $l\in\mathbb{N}\setminus\{1,2,3,4\}$. Then
\begin{equation}\label{eq:muh-filtration-singular-tsisg}
\sum_{i\in S_{l_{n}}}\sqrt{\frac{\mu^{\bm{l}}_{\langle h\rangle}\bigl(K^{\bm{l}}_{\omega^{x}_{1}\ldots\omega^{x}_{n-1}i}\bigr)}{\mu^{\bm{l}}_{\langle h\rangle}\bigl(K^{\bm{l}}_{\omega^{x}_{1}\ldots\omega^{x}_{n-1}}\bigr)}}
	\sqrt{\frac{\refmeas_{\bm{l}}\bigl(K^{\bm{l}}_{\omega^{x}_{1}\ldots\omega^{x}_{n-1}i}\bigr)}{\refmeas_{\bm{l}}\bigl(K^{\bm{l}}_{\omega^{x}_{1}\ldots\omega^{x}_{n-1}}\bigr)}}
	\leq\sqrt{\frac{361}{372}}.
\end{equation}
\end{prop}
\begin{proof}
Set $v:=\omega^{x}_{1}\ldots\omega^{x}_{n-2}$ ($v:=\emptyset$ if $n=2$) and
$w:=\omega^{x}_{1}\ldots\omega^{x}_{n-1}=v\omega^{x}_{n-1}$.
By $h\in\mathcal{H}_{\bm{l},k}\subset\mathcal{H}_{\bm{l},n-2}$,
Proposition \ref{prop:tsisg-harmonic-Vn} and \eqref{eq:EnergyMeas-cell} we have
$h\circ F^{\bm{l}}_{v}\in\mathcal{H}_{\bm{l}^{n-2},0}$,
$h\circ F^{\bm{l}}_{w}
	=(h\circ F^{\bm{l}}_{v})\circ F^{\bm{l}^{n-2}}_{\omega^{x}_{n-1}}
	\in\mathcal{H}_{\bm{l}^{n-1},0}$
and
$\mathcal{E}^{0}(h\circ F^{\bm{l}}_{w}|_{V_{0}},h\circ F^{\bm{l}}_{w}|_{V_{0}})
	=R^{\bm{l}}_{n-1}\mu^{\bm{l}}_{\langle h\rangle}(K^{\bm{l}}_{w})>0$,
and therefore $h\circ F^{\bm{l}}_{w}(V_{0})=\{c-b,c,c+b\}$ for some $b,c\in\mathbb{R}$
with $b>0$ by Lemma \ref{lem:tsisg-harmonic-V1-values} with $\bm{l}^{n-2}$ in place of $\bm{l}$
applied to $i=\omega^{x}_{n-1}\in S_{l_{n-1},1}$ and $u=h\circ F^{\bm{l}}_{v}|_{V_{0}}$.
Then we see from Lemma \ref{lem:tsisg-harmonic-V1-values} with $\bm{l}^{n-1}$ in place of
$\bm{l}$ applied to $i\in S_{l_{n},1}$ and $u=h\circ F^{\bm{l}}_{w}|_{V_{0}}$ that
$h\circ F^{\bm{l}}_{wi}(V_{0})=(h\circ F^{\bm{l}}_{w})\circ F^{\bm{l}^{n-1}}_{i}(V_{0})$
is equal to $\{c_{i}-3a_{l_{n}}b,c_{i},c_{i}+3a_{l_{n}}b\}$
for some $c_{i}\in\mathbb{R}$ for $2(l_{n}-4)$ elements $i$ of $S_{l_{n},1}$ and
to $\{c_{i}-6a_{l_{n}}b,c_{i},c_{i}+6a_{l_{n}}b\}$
for some $c_{i}\in\mathbb{R}$ for the other $l_{n}-4$ elements $i$ of $S_{l_{n},1}$.
It follows by combining this fact with \eqref{eq:tsisg-measure}, \eqref{eq:EnergyMeas-cell},
$h\in\mathcal{H}_{\bm{l},n-1}\subset\mathcal{H}_{\bm{l},n}$, Proposition \ref{prop:tsisg-harmonic-Vn},
\eqref{eq:tsisg-form-0} and \eqref{eq:tsisg-RF-self-sim-form} that
\begin{align*}
&\sum_{i\in S_{l_{n}}}\sqrt{\frac{\mu^{\bm{l}}_{\langle h\rangle}(K^{\bm{l}}_{wi})}{\mu^{\bm{l}}_{\langle h\rangle}(K^{\bm{l}}_{w})}}
	\sqrt{\frac{\refmeas_{\bm{l}}(K^{\bm{l}}_{wi})}{\refmeas_{\bm{l}}(K^{\bm{l}}_{w})}}
	=\sum_{i\in S_{l_{n}}}\sqrt{\frac{\mathcal{E}^{\bm{l}^{n}}(h\circ F^{\bm{l}}_{wi},h\circ F^{\bm{l}}_{wi})}{(r_{l_{n}}\# S_{l_{n}})\mathcal{E}^{\bm{l}^{n-1}}(h\circ F^{\bm{l}}_{w},h\circ F^{\bm{l}}_{w})}} \\
&\leq(l_{n}-4)\frac{2\cdot 3\sqrt{6}a_{l_{n}}b+6\sqrt{6}a_{l_{n}}b}{\sqrt{r_{l_{n}}\# S_{l_{n}}}\cdot\sqrt{6}b}
	+3\sqrt{\frac{\sum_{i\in S_{l_{n}}\setminus S_{l_{n},1}}\mathcal{E}^{\bm{l}^{n}}(h\circ F^{\bm{l}}_{wi},h\circ F^{\bm{l}}_{wi})}{(r_{l_{n}}\# S_{l_{n}})\cdot 6b^{2}}} \\
&=\frac{4(l_{n}-4)}{\sqrt{\# S_{l_{n}}/a_{l_{n}}}}
	+3\sqrt{\frac{r_{l_{n}}\mathcal{E}^{\bm{l}^{n-1}}(h\circ F^{\bm{l}}_{w},h\circ F^{\bm{l}}_{w})-\sum_{i\in S_{l_{n},1}}\mathcal{E}^{\bm{l}^{n}}(h\circ F^{\bm{l}}_{wi},h\circ F^{\bm{l}}_{wi})}{6b^{2}r_{l_{n}}\# S_{l_{n}}}} \\
&=\frac{4(l_{n}-4)}{\sqrt{\# S_{l_{n}}/a_{l_{n}}}}
	+3\sqrt{\frac{6b^{2}r_{l_{n}}-(l_{n}-4)(2\cdot 54a_{l_{n}}^{2}b^{2}+216a_{l_{n}}^{2}b^{2})}{6b^{2}r_{l_{n}}\# S_{l_{n}}}} \\
&=\frac{4(l_{n}-4)}{\sqrt{\# S_{l_{n}}/a_{l_{n}}}}
	+3\sqrt{\frac{a_{l_{n}}^{-1}-6(l_{n}-4)}{\# S_{l_{n}}/a_{l_{n}}}}
	=\frac{4l_{n}-1}{\sqrt{(3l_{n}-3)(6l_{n}+1)}}
	\leq\sqrt{\frac{361}{372}},
\end{align*}
proving \eqref{eq:muh-filtration-singular-tsisg}.
\qed\end{proof}
\begin{proof}[\hspace*{-1.45pt}of Theorem \textup{\ref{thm:tsisg-sing}}]
Let $k\in\mathbb{N}\cup\{0\}$ and $h\in\mathcal{H}_{\bm{l},k}$. In view of
Lemmas \ref{lem:tsisg-harmonic-approx} and \ref{lem:tsisg-sing-lin-cont} it suffices to
prove, for any such $k$ and $h$, that $\mu^{\bm{l}}_{\langle h\rangle}\perp \refmeas_{\bm{l}}$,
which is obvious if $\mu^{\bm{l}}_{\langle h\rangle}(K^{\bm{l}})=0$.
Assume that $\mu^{\bm{l}}_{\langle h\rangle}(K^{\bm{l}})>0$, set
$(\Omega,\mathscr{F},\mathbb{P}):=(K^{\bm{l}},\mathscr{B}(K^{\bm{l}}),\refmeas_{\bm{l}})$,
let $\{\mathscr{F}_{n}\}_{n=0}^{\infty}$ denote the non-decreasing sequence of $\sigma$-algebras
in $\Omega$ with $\bigcup_{n=0}^{\infty}\mathscr{F}_{n}$ generating $\mathscr{F}$
as defined in Lemma \ref{lem:prob-meas-filtration-singular-tsisg}, and set
$\widetilde{\mathbb{P}}:=\mu^{\bm{l}}_{\langle h\rangle}(K^{\bm{l}})^{-1}\mu^{\bm{l}}_{\langle h\rangle}$,
so that $\mathbb{P}(K^{\bm{l}}_{w})>0$ for any $w\in W^{\bm{l}}_{*}$ and
$\mathbb{P}(V^{\bm{l}}_{*})=\widetilde{\mathbb{P}}(V^{\bm{l}}_{*})=0$ by
\eqref{eq:tsisg-measure} and Proposition \ref{prop:EnergyMeas-cell}. In particular,
$\widetilde{\mathbb{P}}|_{\mathscr{F}_{n}}\ll\mathbb{P}|_{\mathscr{F}_{n}}$
for any $n\in\mathbb{N}\cup\{0\}$, and define
$\alpha_{n}\in L^{1}(\Omega,\mathscr{F}_{n},\mathbb{P}|_{\mathscr{F}_{n}})$
by \eqref{eq:prob-meas-filtration-singular-alphan} for each $n\in\mathbb{N}$.
Now let $\omega^{x}=(\omega^{x}_{n})_{n=1}^{\infty}\in\prod_{n=1}^{\infty}S_{l_{n}}$
for $x\in K^{\bm{l}}\setminus V^{\bm{l}}_{*}$ and $S_{l,1}\subset S_{l}$ for
$l\in\mathbb{N}\setminus\{1,2,3,4\}$ be as in Proposition \ref{prop:muh-filtration-singular-tsisg}.
Then by \eqref{eq:tsisg-measure}, the $\mathbb{P}$-a.s.\ defined Borel measurable maps
$K^{\bm{l}}\setminus V^{\bm{l}}_{*}\ni x\mapsto\omega^{x}_{n}\in S_{l_{n}}$,
$n\in\mathbb{N}$, form a sequence of independent random variables on
$(\Omega,\mathscr{F},\mathbb{P})$ and satisfy
\begin{equation*}
\sum_{n=1}^{\infty}\mathbb{P}\bigl(\{x\in K^{\bm{l}}\setminus V^{\bm{l}}_{*}\mid\omega^{x}_{n}\in S_{l_{n},1}\}\bigr)
	=\sum_{n=1}^{\infty}\frac{\# S_{l_{n},1}}{\# S_{l_{n}}}
	=\sum_{n=1}^{\infty}\frac{3l_{n}-12}{3l_{n}-3}
	\geq\sum_{n=1}^{\infty}\frac{1}{4}=\infty,
\end{equation*}
and hence the second Borel--Cantelli lemma implies that
\begin{equation}\label{eq:limsup-Sl1-almost-surely}
\#\{n\in\mathbb{N}\mid\omega^{x}_{n}\in S_{l_{n},1}\}=\infty
	\qquad\textrm{for $\mathbb{P}$-a.e.\ $x\in K^{\bm{l}}\setminus V^{\bm{l}}_{*}$.}
\end{equation}
On the other hand, for each $x\in K^{\bm{l}}\setminus V^{\bm{l}}_{*}$,
Lemma \ref{lem:prob-meas-filtration-singular-tsisg} and
Proposition \ref{prop:muh-filtration-singular-tsisg} imply that
$\mathbb{E}[\sqrt{\alpha_{n}}\mid\mathscr{F}_{n-1}](x)=0$ for any $n\in\mathbb{N}$ with
$\mu^{\bm{l}}_{\langle h\rangle}\bigl(K^{\bm{l}}_{\omega^{x}_{1}\ldots\omega^{x}_{n-1}}\bigr)=0$
and that
\begin{align*}
\mathbb{E}[\sqrt{\alpha_{n}}\mid\mathscr{F}_{n-1}](x)
	=\sum_{i\in S_{l_{n}}}\sqrt{\frac{\mu^{\bm{l}}_{\langle h\rangle}\bigl(K^{\bm{l}}_{\omega^{x}_{1}\ldots\omega^{x}_{n-1}i}\bigr)}{\mu^{\bm{l}}_{\langle h\rangle}\bigl(K^{\bm{l}}_{\omega^{x}_{1}\ldots\omega^{x}_{n-1}}\bigr)}}
		\sqrt{\frac{\refmeas_{\bm{l}}\bigl(K^{\bm{l}}_{\omega^{x}_{1}\ldots\omega^{x}_{n-1}i}\bigr)}{\refmeas_{\bm{l}}\bigl(K^{\bm{l}}_{\omega^{x}_{1}\ldots\omega^{x}_{n-1}}\bigr)}}
	\leq\sqrt{\frac{361}{372}}
\end{align*}
for any $n\in\mathbb{N}\cap[k+2,\infty)$ with
$\mu^{\bm{l}}_{\langle h\rangle}\bigl(K^{\bm{l}}_{\omega^{x}_{1}\ldots\omega^{x}_{n-1}}\bigr)>0$
and $\omega^{x}_{n-1}\in S_{l_{n-1},1}$, whence
\begin{equation}\label{eq:muh-filtration-singular-tsisg-sum}
\sum_{n=1}^{\infty}\bigl(1-\mathbb{E}[\sqrt{\alpha_{n}}\mid\mathscr{F}_{n-1}](x)\bigr)
	\geq\delta\#\{n\in\mathbb{N}\cap[k+2,\infty)\mid\omega^{x}_{n-1}\in S_{l_{n-1},1}\},
\end{equation}
where $\delta:=1-\sqrt{\frac{361}{372}}\in(0,1)$. Combining
\eqref{eq:limsup-Sl1-almost-surely} and \eqref{eq:muh-filtration-singular-tsisg-sum},
we obtain \eqref{eq:prob-meas-filtration-singular}, so that
Theorem \ref{thm:prob-meas-filtration-singular} is applicable and yields
$\widetilde{\mathbb{P}}\perp\mathbb{P}$, namely
$\mu^{\bm{l}}_{\langle h\rangle}\perp \refmeas_{\bm{l}}$.
\qed\end{proof}
%
\section{Realizing arbitrarily slow decay rates of $\Psi(r)/r^{2}$}\label{sec:realize-given-Psi}
In this last section, we show that an arbitrarily slow decay rate of $\Psi(r)/r^{2}$
for a homeomorphism $\Psi\colon[0,\infty)\to[0,\infty)$ satisfying \eqref{eq:Psi-ass}
and \eqref{eq:case-nonGauss} can be realized by $\Psi_{\bm{l}}$ (recall Definition \ref{dfn:tsisg-Psi})
for some $\bm{l}=(l_{n})_{n=1}^{\infty}\in(\mathbb{N}\setminus\{1,2,3,4\})^{\mathbb{N}}$.
We achieve this by providing in Theorem \ref{thm:realize-given-Psi}
a simple sufficient condition for $\Psi$ to be comparable to $\Psi_{\bm{l}}$ for some
$\bm{l}=(l_{n})_{n=1}^{\infty}\in(\mathbb{N}\setminus\{1,2,3,4\})^{\mathbb{N}}$ with
$\sum_{n=1}^{\infty}l_{n}^{-1}<\infty$ and proving in Proposition \ref{prop:tsisg-Psi-arbitrarily-slow}
that the decay rate of $\Psi(r)/r^{2}$ for such $\Psi$ can be arbitrarily slow.
We also give criteria for verifying this sufficient condition for concrete
examples of $\Psi$ in Proposition \ref{prop:realize-given-Psi} and apply them
to the case where $\Psi(r)/r^{2}$ is a multiple composition of the function
$r\mapsto 1/\log(e-1+(r\wedge 1)^{-1})$ in Example \ref{exmp:realize-given-Psi}.
\begin{thm}\label{thm:realize-given-Psi}
Let $\eta\colon[0,1]\to[0,1]$ be a homeomorphism with $\eta(0)=0$ such that
\begin{equation}\label{eq:realize-given-Psi-eta}
\sum_{n=1}^{\infty}\frac{\eta^{-1}(2^{-n})}{\eta^{-1}(2^{1-n})}<\infty,
\end{equation}
and define a homeomorphism $\Psi_{\eta}\colon[0,\infty)\to[0,\infty)$ by
$\Psi_{\eta}(r):=r^{2}\eta(r\wedge 1)$. Then there exists
$\bm{l}=(l_{n})_{n=1}^{\infty}\in(\mathbb{N}\setminus\{1,2,3,4\})^{\mathbb{N}}$
with $\sum_{n=1}^{\infty}l_{n}^{-1}<\infty$ such that
$\Psi_{\eta}(r)/\Psi_{\bm{l}}(r)\in[c^{-1},c]$ for any $r\in(0,\infty)$
for some $c\in[1,\infty)$, and consequently,
$(K^{\bm{l}},\refmet_{\bm{l}},\refmeas_{\bm{l}},\mathcal{E}^{\bm{l}},\mathcal{F}_{\bm{l}})$
satisfies \hyperlink{fHKE}{$\textup{fHKE}(\Psi_{\eta})$}.
\end{thm}
\begin{proof}
Set $c_{\eta}:=\inf_{n\in\mathbb{N}}\eta^{-1}(2^{1-n})/\eta^{-1}(2^{-n})$,
so that $c_{\eta}\in(1,\infty)$ since the sequence
$\{\eta^{-1}(2^{1-n})/\eta^{-1}(2^{-n})\}_{n=1}^{\infty}$
is $(1,\infty)$-valued and tends to $\infty$ by \eqref{eq:realize-given-Psi-eta}.
Then for any $r,R\in(0,1]$ with $r\leq R$, taking $j,k\in\mathbb{N}$ such that
$\eta(r)\in(2^{-k},2^{1-k}]$ and $\eta(R)\in(2^{-j},2^{1-j}]$,
we have $j\leq k$, hence
\begin{equation*}
\frac{R}{r}=\frac{\eta^{-1}(\eta(R))}{\eta^{-1}(\eta(r))}
	\geq\frac{\eta^{-1}(2^{-j})}{\eta^{-1}(2^{1-k})}\vee 1
	\geq c_{\eta}^{k-j-1}=2^{(k-j-1)/\beta_{\eta}}
	\geq\biggl(\frac{\eta(R)}{4\eta(r)}\biggr)^{1/\beta_{\eta}}
\end{equation*}
by the definition of $c_{\eta}$, where $\beta_{\eta}:=(\log_{2}c_{\eta})^{-1}\in(0,\infty)$, and therefore
\begin{equation}\label{eq:realize-given-Psi-eta-doubling}
\frac{\eta(R)}{\eta(r)}\leq 4\biggl(\frac{R}{r}\biggr)^{\beta_{\eta}}.
\end{equation}

Recalling that $\lim_{n\to\infty}\eta^{-1}(2^{1-n})/\eta^{-1}(2^{-n})=\infty$
by \eqref{eq:realize-given-Psi-eta}, choose $n_{0}\in\mathbb{N}$ so that
$\eta^{-1}(2^{1-n})/\eta^{-1}(2^{-n})\geq 5$ for any $n\in\mathbb{N}$
with $n\geq n_{0}$, set $l_{0}:=1$, and define
$\bm{l}=(l_{n})_{n=1}^{\infty}\in\mathbb{N}^{\mathbb{N}}$ inductively by
\begin{equation}\label{eq:realize-given-Psi-ln}
l_{n}:=\bigg\lfloor\frac{\eta^{-1}(2^{-n_{0}})}{(l_{0}\cdots l_{n-1})\eta^{-1}(2^{-n-n_{0}})}\biggr\rfloor,
	\qquad n\in\mathbb{N}.
\end{equation}
Then an induction on $n$ based on \eqref{eq:realize-given-Psi-ln}
and the choice of $n_{0}$ immediately shows that
$\bm{l}=(l_{n})_{n=1}^{\infty}\in(\mathbb{N}\setminus\{1,2,3,4\})^{\mathbb{N}}$
and that for any $n\in\mathbb{N}\cup\{0\}$,
\begin{equation}\label{eq:realize-given-Psi-ln-bounds}
\frac{\eta^{-1}(2^{-n-n_{0}})}{\eta^{-1}(2^{-n_{0}})}
	\leq\frac{1}{L^{\bm{l}}_{n}}
	\leq\frac{6}{5}\frac{\eta^{-1}(2^{-n-n_{0}})}{\eta^{-1}(2^{-n_{0}})},
\end{equation}
which together with \eqref{eq:realize-given-Psi-eta} implies in particular that
\begin{equation}\label{eq:realize-given-Psi-ln-ell1}
\sum_{n=1}^{\infty}l_{n}^{-1}
	=\sum_{n=1}^{\infty}\frac{1/L^{\bm{l}}_{n}}{1/L^{\bm{l}}_{n-1}}
	\leq\frac{6}{5}\sum_{n=1}^{\infty}\frac{\eta^{-1}(2^{-n-n_{0}})}{\eta^{-1}(2^{1-n-n_{0}})}
	<\infty.
\end{equation}

We claim that $\Psi_{\eta}(r)/\Psi_{\bm{l}}(r)\in[c^{-1},c]$ for any $r\in(0,\infty)$
for some $c\in[1,\infty)$. Indeed, recalling Definition \ref{dfn:tsisg-Psi}, we have
$\beta_{\bm{l},0}=\inf_{n\in\mathbb{N}}\beta_{l_{n}}=2$ by \eqref{eq:realize-given-Psi-ln-ell1},
hence $\Psi_{\eta}(r)/\Psi_{\bm{l}}(r)=r^{2}/r^{\beta_{\bm{l},0}}=1$ for any $r\in[1,\infty)$,
and also see for any $n\in\mathbb{N}$ that
\begin{equation}\label{eq:realize-given-Psi-comparable-eta}
2^{-n-n_{0}}\leq\eta\biggl(\frac{\eta^{-1}(2^{-n-n_{0}})}{\eta^{-1}(2^{-n_{0}})}\biggr)
	\leq\eta(1/L^{\bm{l}}_{n})
	\leq\eta\biggl(\frac{6}{5}\frac{\eta^{-1}(2^{-n-n_{0}})}{\eta^{-1}(2^{-n_{0}})}\biggr)
	\leq c2^{-n}
\end{equation}
by \eqref{eq:realize-given-Psi-ln-bounds} and \eqref{eq:realize-given-Psi-eta-doubling},
where $c:=2^{2-n_{0}}\bigl(\frac{6}{5}/\eta^{-1}(2^{-n_{0}})\bigr)^{\beta_{\eta}}$, and thus that
\begin{equation}\label{eq:realize-given-Psi-comparable}
\frac{\Psi_{\eta}(1/L^{\bm{l}}_{n})}{\Psi_{\bm{l}}(1/L^{\bm{l}}_{n})}
	=\frac{T^{\bm{l}}_{n}\eta(1/L^{\bm{l}}_{n})}{(L^{\bm{l}}_{n})^{2}}
	=2^{n}\eta(1/L^{\bm{l}}_{n})\prod_{k=1}^{n}\biggl(1-\frac{5}{6}l_{k}^{-1}-\frac{1}{6}l_{k}^{-2}\biggr)
	\in[c',c],
\end{equation}
where $c':=2^{-n_{0}}\prod_{k=1}^{\infty}(1-\frac{5}{6}l_{k}^{-1}-\frac{1}{6}l_{k}^{-2})\in(0,1)$
by \eqref{eq:realize-given-Psi-ln-ell1}. Now for any $n\in\mathbb{N}$
and any $s\in[1,l_{n}]$, by \eqref{eq:tsisg-Psi} we have
\begin{equation}\label{eq:tsisg-Psi-compare-square}
\frac{\Psi_{\bm{l}}(s/L^{\bm{l}}_{n})}{s^{2}\Psi_{\bm{l}}(1/L^{\bm{l}}_{n})}
	=\frac{1}{s^{2}}\biggl(1+\frac{3l_{n}-4}{l_{n}-1}(s-1)\biggr)
	\biggl(1+\frac{\frac{2}{3}l_{n}-\frac{8}{9}}{l_{n}-1}(s-1)\biggr)
	\in[1,2),
\end{equation}
and it follows from
$\eta(s/L^{\bm{l}}_{n})\in\bigl[\eta(1/L^{\bm{l}}_{n}),\eta(1/L^{\bm{l}}_{n-1})\bigr]$,
\eqref{eq:realize-given-Psi-comparable-eta}, \eqref{eq:tsisg-Psi-compare-square}
and \eqref{eq:realize-given-Psi-comparable} that
\begin{align*}
\frac{\Psi_{\eta}(s/L^{\bm{l}}_{n})}{\Psi_{\bm{l}}(s/L^{\bm{l}}_{n})}
	&=\frac{\Psi_{\eta}(s/L^{\bm{l}}_{n})}{\Psi_{\eta}(1/L^{\bm{l}}_{n})}
		\frac{\Psi_{\bm{l}}(1/L^{\bm{l}}_{n})}{\Psi_{\bm{l}}(s/L^{\bm{l}}_{n})}
		\frac{\Psi_{\eta}(1/L^{\bm{l}}_{n})}{\Psi_{\bm{l}}(1/L^{\bm{l}}_{n})} \\
&=\frac{\eta(s/L^{\bm{l}}_{n})}{\eta(1/L^{\bm{l}}_{n})}
	\frac{s^{2}\Psi_{\bm{l}}(1/L^{\bm{l}}_{n})}{\Psi_{\bm{l}}(s/L^{\bm{l}}_{n})}
	\frac{\Psi_{\eta}(1/L^{\bm{l}}_{n})}{\Psi_{\bm{l}}(1/L^{\bm{l}}_{n})}
	\in[c'/2,(c^{2}\vee c)2^{n_{0}+1}],
\end{align*}
proving that $\Psi_{\eta}(r)/\Psi_{\bm{l}}(r)\in[c'/2,(c^{2}\vee c)2^{n_{0}+1}]$
for any $r\in(0,\infty)$. Lastly, combining this result with Lemma \ref{lem:tsisg-Psi},
Theorem \ref{thm:tsisg-fHKE} and Remark \ref{rmk:HK-singular}-\ref{it:fHKE-stable-Psi} shows that
$(K^{\bm{l}},\refmet_{\bm{l}},\refmeas_{\bm{l}},\mathcal{E}^{\bm{l}},\mathcal{F}_{\bm{l}})$
satisfies \hyperlink{fHKE}{$\textup{fHKE}(\Psi_{\eta})$}.
\qed\end{proof}

The decay rate of $\Psi_{\eta}(r)/r^{2}=\eta(r\wedge 1)$ for $\eta$ as in Theorem \ref{thm:realize-given-Psi}
can be arbitrarily slow in the sense stated in the following proposition.
\begin{prop}\label{prop:tsisg-Psi-arbitrarily-slow}
Let $\Psi\colon[0,\infty)\to[0,\infty)$ be a homeomorphism satisfying \eqref{eq:case-nonGauss}.
Then there exists a homeomorphism $\eta\colon[0,1]\to[0,1]$ with the properties
$\eta(0)=0$ and \eqref{eq:realize-given-Psi-eta} such that
$\eta(r)\geq c\Psi(r)/r^{2}$ for any $r\in(0,1]$ for some $c\in(0,\infty)$.
\end{prop}
\begin{proof}
Noting \eqref{eq:case-nonGauss}, define $\eta_{0}\colon[0,1]\to[0,\infty)$
by $\eta_{0}(r):=\sup_{s\in(0,r]}\Psi(s)/s^{2}$ ($\eta_{0}(0):=0$), so that
$\eta_{0}$ is continuous and non-decreasing and $\eta_{0}((0,1])\subset(0,\infty)$,
and set $s_{n}:=\max\eta_{0}^{-1}(2^{-n}\eta_{0}(1))$ for $n\in\mathbb{N}\cup\{0\}$,
so that $s_{0}=1$, $0<s_{n}<s_{n-1}$ for any $n\in\mathbb{N}$ and $\lim_{n\to\infty}s_{n}=0$.
Define a homeomorphism $\eta\colon[0,1]\to[0,1]$ by
\begin{equation}\label{eq:tsisg-Psi-arbitrarily-slow}
\eta(r):=\biggl(1+\frac{r-2^{-n^{2}}s_{n}}{2^{-(n-1)^{2}}s_{n-1}-2^{-n^{2}}s_{n}}\biggr)2^{-n}
\end{equation}
for $n\in\mathbb{N}$ and $r\in[2^{-n^{2}}s_{n},2^{-(n-1)^{2}}s_{n-1}]$ and $\eta(0):=0$.
Then since $\eta^{-1}(2^{1-n})=2^{-(n-1)^{2}}s_{n-1}$ and $0<s_{n}<s_{n-1}$ for any $n\in\mathbb{N}$,
\begin{equation*}
\sum_{n=1}^{\infty}\frac{\eta^{-1}(2^{-n})}{\eta^{-1}(2^{1-n})}
	=\sum_{n=1}^{\infty}\frac{2^{-n^{2}}s_{n}}{2^{-(n-1)^{2}}s_{n-1}}
	\leq\sum_{n=1}^{\infty}2^{1-2n}=\frac{2}{3}<\infty,
\end{equation*}
namely $\eta$ satisfies \eqref{eq:realize-given-Psi-eta},
and for any $n\in\mathbb{N}$ and any $r\in[s_{n},s_{n-1}]$ we have
\begin{equation*}
\eta(r)\geq\eta(s_{n})\geq\eta(2^{-n^{2}}s_{n})=2^{-n}=\frac{\eta_{0}(s_{n-1})}{2\eta_{0}(1)}
	\geq\frac{\eta_{0}(r)}{2\eta_{0}(1)}\geq\frac{\Psi(r)/r^{2}}{2\eta_{0}(1)},
\end{equation*}
i.e., $\eta(r)\geq c\Psi(r)/r^{2}$ with $c:=(2\eta_{0}(1))^{-1}\in(0,\infty)$ for any $r\in(0,1]$.
\qed\end{proof}

We conclude this paper with the following proposition, which gives criteria for verifying
\eqref{eq:realize-given-Psi-eta} for concrete homeomorphisms $\eta\colon[0,1]\to[0,1]$
with $\eta(0)=0$, and some applications of it to $\eta(r)=1/\log(e-1+r^{-1})$
in Example \ref{exmp:realize-given-Psi} below.
\begin{prop}\label{prop:realize-given-Psi}
Let $\eta\colon[0,1]\to[0,1]$ be a homeomorphism with $\eta(0)=0$, let $\delta\in[0,\infty)$,
$\alpha,\beta\in(0,\infty)$ and assume that there exists $c\in(0,\infty)$ such that
\begin{equation}\label{eq:realize-given-Psi-eta-verify}
\frac{\eta(R)}{\eta(r)}\leq 1+\delta+\frac{c(R/r)^{\beta}}{(\log(e-1+R^{-1}))^{\alpha}}
	\mspace{22mu}\textrm{for any $r,R\in(0,1]$ with $r\leq R$.}
\end{equation}
\begin{enumerate}[label=\textup{(\arabic*)},align=left,leftmargin=*]
\item\label{it:realize-given-Psi-eta-verify}If $\delta<1$ and $\beta<\alpha$,
	then $\eta$ satisfies \eqref{eq:realize-given-Psi-eta}.
\item\label{it:realize-given-Psi-eta-composition}Let $\widetilde{\eta}\colon[0,1]\to[0,1]$
	be a homeomorphism with $\widetilde{\eta}(0)=0$, let $\widetilde{\delta}\in[0,1)$
	and assume that there exist $\widetilde{\alpha},\widetilde{c}\in(0,\infty)$
	such that $\widetilde{\eta}$ satisfies \eqref{eq:realize-given-Psi-eta-verify}
	with $\widetilde{\delta},\widetilde{\alpha},1,\widetilde{c}$
	in place of $\delta,\alpha,\beta,c$. Then $\widetilde{\eta}\circ\eta$ satisfies
	\eqref{eq:realize-given-Psi-eta-verify} with $\frac{1}{2}(1+\widetilde{\delta}),c'$
	in place of $\delta,c$ for some $c'\in(0,\infty)$. In particular, if $\beta<\alpha$,
	then $\widetilde{\eta}\circ\eta$ satisfies \eqref{eq:realize-given-Psi-eta}.
\end{enumerate}
\end{prop}
\begin{proof}
\begin{enumerate}[label=\textup{(\arabic*)},align=left,leftmargin=*]
\item Set $s_{n}:=\eta^{-1}(2^{-n})$ for $n\in\mathbb{N}\cup\{0\}$.
	For each $n\in\mathbb{N}$, we see from $\eta(s_{n-1})/\eta(s_{n})=2$,
	\eqref{eq:realize-given-Psi-eta-verify} with $(r,R)=(r_{n},r_{n-1})$ and $\delta<1$ that
	\begin{equation}\label{eq:realize-given-Psi-eta-verify-proof1}
	\frac{s_{n}}{s_{n-1}}\leq\frac{c^{1/\beta}(1-\delta)^{-1/\beta}}{(\log(e-1+s_{n-1}^{-1}))^{\alpha/\beta}},
	\end{equation}
	and from $\eta(1)/\eta(s_{n-1})=2^{n-1}$, \eqref{eq:realize-given-Psi-eta-verify}
	with $(r,R)=(r_{n-1},1)$ and $\delta<1$ that
	$2^{n-1}\leq 1+\delta+cs_{n-1}^{-\beta}\leq(2+c)s_{n-1}^{-\beta}$,
	whence, provided $n\geq 2+2\log_{2}(2+c)$,
	\begin{equation}\label{eq:realize-given-Psi-eta-verify-proof2}
	\log(e-1+s_{n-1}^{-1})\geq\log(s_{n-1}^{-1})
		\geq\frac{\log 2}{\beta}(n-1-\log_{2}(2+c))
		\geq\frac{\log 2}{2\beta}n.
	\end{equation}
	It follows from \eqref{eq:realize-given-Psi-eta-verify-proof1},
	\eqref{eq:realize-given-Psi-eta-verify-proof2} and $\alpha/\beta>1$ that
	\begin{equation*}
	\sum_{n=1}^{\infty}\frac{\eta^{-1}(2^{-n})}{\eta^{-1}(2^{1-n})}
		=\sum_{n=1}^{\infty}\frac{s_{n}}{s_{n-1}}
		\leq\sum_{n=1}^{n_{c}-1}\frac{s_{n}}{s_{n-1}}
			+\sum_{n=n_{c}}^{\infty}\frac{c^{1/\beta}(2\beta/\log 2)^{\alpha/\beta}}{(1-\delta)^{1/\beta}n^{\alpha/\beta}}
		<\infty,
	\end{equation*}
	where $n_{c}:=3+\lfloor 2\log_{2}(2+c)\rfloor$, proving \eqref{eq:realize-given-Psi-eta}.
\item Set $\widetilde{r}:=\eta^{-1}\bigl(\exp\bigl(-(2\widetilde{c}(1+\delta)/(1-\widetilde{\delta}))^{1/\widetilde{\alpha}}\bigr)\bigr)$
	and let $r,R\in(0,1]$ satisfy $r\leq R$. By \eqref{eq:realize-given-Psi-eta-verify} for
	$\widetilde{\eta}$ and $\eta$ and $(\log(e-1+\eta(R)^{-1}))^{-\widetilde{\alpha}}\leq 1$ we have
	\begin{equation}\label{eq:realize-given-Psi-eta-composition}
	\begin{split}
	\frac{\widetilde{\eta}\circ\eta(R)}{\widetilde{\eta}\circ\eta(r)}
		&\leq 1+\widetilde{\delta}+\frac{\widetilde{c}}{(\log(e-1+\eta(R)^{-1}))^{\widetilde{\alpha}}}\frac{\eta(R)}{\eta(r)} \\
	&\leq 1+\widetilde{\delta}+\frac{\widetilde{c}(1+\delta)}{(\log(e-1+\eta(R)^{-1}))^{\widetilde{\alpha}}}
		+\frac{\widetilde{c}c(R/r)^{\beta}}{(\log(e-1+R^{-1}))^{\alpha}}.
	\end{split}
	\end{equation}
	If $R\leq\widetilde{r}$, then
	$\widetilde{c}(1+\delta)/(\log(e-1+\eta(R)^{-1}))^{\widetilde{\alpha}}\leq\frac{1}{2}(1-\widetilde{\delta})$
	by the definition of $\widetilde{r}$ and hence \eqref{eq:realize-given-Psi-eta-composition}
	yields \eqref{eq:realize-given-Psi-eta-verify} with
	$\widetilde{\eta}\circ\eta,\frac{1}{2}(1+\widetilde{\delta}),\widetilde{c}c$
	in place of $\eta,\delta,c$, whereas if $R>\widetilde{r}$, then we see from
	\eqref{eq:realize-given-Psi-eta-composition}, $\widetilde{\delta}<1$ and
	$(\log(e-1+\eta(R)^{-1}))^{-\widetilde{\alpha}}\leq 1
		\leq(\log(e-1+\widetilde{r}^{-1})/\log(e-1+R^{-1}))^{\alpha}\wedge(R/r)^{\beta}$
	that \eqref{eq:realize-given-Psi-eta-verify} with
	$\widetilde{\eta}\circ\eta,\frac{1}{2}(1+\widetilde{\delta}),c'$ in place of $\eta,\delta,c$ holds,
	where $c':=\widetilde{c}(1+\delta)(\log(e-1+\widetilde{r}^{-1}))^{\alpha}+\widetilde{c}c$.
	In particular, if $\beta<\alpha$, then $\widetilde{\eta}\circ\eta$ satisfies
	\eqref{eq:realize-given-Psi-eta} by $\frac{1}{2}(1+\widetilde{\delta})<1$ and
	\ref{it:realize-given-Psi-eta-verify}.
\qed\end{enumerate}
\end{proof}
\begin{exmp}\label{exmp:realize-given-Psi}
Define homeomorphisms $\eta_{k}\colon[0,1]\to[0,1]$, $k\in\mathbb{N}$, inductively by
\begin{equation}\label{eq:realize-given-Psi-log}
\eta_{1}(r):=\frac{1}{\log(e-1+r^{-1})} \mspace{10mu} \textrm{($\eta_{1}(0):=0$)}
	\quad\textrm{and}\quad \eta_{k+1}:=\eta_{1}\circ\eta_{k}, \mspace{10mu} k\in\mathbb{N}.
\end{equation}\vspace*{-1pt}%
Then $\eta_{k}$ satisfies \eqref{eq:realize-given-Psi-eta-verify} with $\delta=\frac{1}{2}$
and $\alpha=1$ for some $c\in(0,\infty)$ for \emph{any} $\beta\in(0,\infty)$ and any
$k\in\mathbb{N}$. Indeed, this follows by a straightforward induction on $k$ based on
Proposition \ref{prop:realize-given-Psi}-\ref{it:realize-given-Psi-eta-composition},
which is applicable with $\eta=\eta_{k}$ and $\widetilde{\eta}=\eta_{1}$ since
$\eta_{1}$ is easily seen to satisfy \eqref{eq:realize-given-Psi-eta-verify} with
$\delta=0$, $\alpha=1$ and $c=(e\beta)^{-1}$ for \emph{any} $\beta\in(0,\infty)$
as follows: for any $r,R\in(0,1]$ with $r\leq R$,
\begin{equation}\label{eq:realize-given-Psi-log-proof}
\begin{split}
\frac{\eta_{1}(R)}{\eta_{1}(r)}&=1+\frac{\log\dfrac{e-1+r^{-1}}{e-1+R^{-1}}}{\log(e-1+R^{-1})}
	=1+\frac{\log\dfrac{R}{r}+\log\dfrac{1+(e-1)r}{1+(e-1)R}}{\log(e-1+R^{-1})} \\
&\leq 1+\frac{\beta^{-1}\log((R/r)^{\beta})}{\log(e-1+R^{-1})}
	\leq 1+\frac{(e\beta)^{-1}(R/r)^{\beta}}{\log(e-1+R^{-1})}.
\end{split}
\end{equation}
As a consequence, for each $k\in\mathbb{N}$, recalling that
$\Psi_{\eta_{k}}\colon[0,\infty)\to[0,\infty)$ is defined by
$\Psi_{\eta_{k}}(r):=r^{2}\eta_{k}(r\wedge 1)$, we conclude from
Proposition \ref{prop:realize-given-Psi}-\ref{it:realize-given-Psi-eta-verify}
that $\eta_{k}$ satisfies \eqref{eq:realize-given-Psi-eta},
thus from Theorem \ref{thm:realize-given-Psi} that there exists
$\bm{l}_{k}=(l_{k,n})_{n=1}^{\infty}\in(\mathbb{N}\setminus\{1,2,3,4\})^{\mathbb{N}}$
with $\sum_{n=1}^{\infty}l_{k,n}^{-1}<\infty$ such that
$\Psi_{\eta_{k}}(r)/\Psi_{\bm{l}_{k}}(r)\in[c_{k}^{-1},c_{k}]$ for any $r\in(0,\infty)$
for some $c_{k}\in[1,\infty)$, and thereby that
$(K^{\bm{l}_{k}},\refmet_{\bm{l}_{k}},\refmeas_{\bm{l}_{k}},\mathcal{E}^{\bm{l}_{k}},\mathcal{F}_{\bm{l}_{k}})$
satisfies \hyperlink{fHKE}{$\textup{fHKE}(\Psi_{\eta_{k}})$}.
\end{exmp}
\end{document}